\documentclass[a4paper,11pt,reqno]{amsart} 
\usepackage{a4wide}
\usepackage{xcolor}

\usepackage[leqno]{amsmath}

\makeatletter
\newcommand{\leqnomode}{\tagsleft@true}
\newcommand{\reqnomode}{\tagsleft@false}
\makeatother

\usepackage[T1]{fontenc}
\usepackage[utf8]{inputenc}
\usepackage{fourier}

\usepackage{amsmath, amsthm, mathtools}
\usepackage{bm} 
\usepackage{mathrsfs} 

\usepackage{graphicx}
\usepackage{subcaption} 
\usepackage{booktabs}   
\usepackage{tikz}
\usetikzlibrary{arrows.meta, positioning}

\usepackage[colorlinks=true,citecolor=blue,linkcolor=blue,urlcolor=blue]{hyperref}

\usepackage{amsmath} 
\usepackage{amsthm}
\usepackage{amsfonts}
\usepackage{mathtools}
\usepackage{enumitem}
\usepackage{cite}
\usepackage{hyperref}
\usepackage{cleveref}
\usepackage{latexsym}

\theoremstyle{plain}
\newtheorem{theorem}{Theorem}[section]
\newtheorem{lemma}[theorem]{Lemma}
\newtheorem{prop}[theorem]{Proposition}

\theoremstyle{definition}
\newtheorem{definition}[theorem]{Definition}

\newtheorem{remark}[theorem]{Remark}

\newcommand{\R}{\mathbb{R}}

\newcommand{\J}{\mathcal{J}}

\newcommand{\I}{\mathcal{I}}

\newcommand{\osc}{\mathrm{osc}}

\newcommand{\1}{\pmb{1}}

\numberwithin{equation}{section}

\begin{document}

\title[Regularity Superlinear Nonlocal Hamilton-Jacobi]
{Improved Regularity for Nonlocal Hamilton--Jacobi Equations with Superlinear Hamiltonians}

\author[]{Adina Ciomaga}
\address{Adina Ciomaga: 
Université Paris Cité, CNRS, Sorbonne Université, Laboratoire Jacques-Louis Lions (LJLL), F-75006 Paris, France; 
O. Mayer Mathematics Institute, Romanian Academy, Ia\c si, 700506 Ia\c si, Romania.}
\email{\tt adina@math.univ-paris-diderot.fr}

\author[]{Tr\'i Minh L\^e}
\address{Tr\'i Minh L\^e:
Fakult\"at f\"ur Mathematik, Universit\"at Wien, Oskar-Morgenstern-Platz 1, 1090 Wien.}
\email{\tt tri.minh.le@univie.ac.at}

\author[]{Olivier Ley}
\address{Olivier Ley:
Univ Rennes, INSA Rennes, CNRS, IRMAR -- UMR 6625, F-35000 Rennes, France.}
\email{\tt olivier.ley@insa-rennes.fr}

\author[]{Erwin Topp}
\address{Erwin Topp:
Instituto de Matem\'atica, Universidade Federal do Rio de Janeiro, 
Rio de Janeiro -- RJ, 21941-909, Brazil.}
\email{\tt etopp@im.ufrj.br}

\date{\today}

\maketitle

\begin{abstract}
We investigate the regularity of viscosity solutions to a class of nonlocal Hamilton--Jacobi equations driven by $x$-dependent integro--differential operators and coercive superlinear Hamiltonians.
We first establish H\"older regularity for bounded viscosity solutions under general structural and continuity assumptions on the underlying L\'evy measures, without imposing any ellipticity condition on the nonlocal operator. The H\"older exponent is given explicitly in terms of the order $\sigma\in(0,2)$ of the operator and the growth exponent $m>1$ of the Hamiltonian. In particular, our approach applies to arbitrary superlinear Hamiltonians, including the delicate regime $1<m<\sigma<2$, and yields an improved regularity exponent when $\sigma\in(1,2)$.
Assuming in addition a weak ellipticity condition on the nonlocal operator, we prove that viscosity solutions are globally Lipschitz continuous. The proof combines the H\"older regularity supplied by the coercive Hamiltonian with the regularizing effect of the nonlocal diffusion through an Ishii--Lions argument, allowing us to treat Hamiltonians with arbitrary superlinear growth.
Finally, we provide a counterexample showing that, in the absence of ellipticity, Lipschitz regularity may fail if the spatial dependence of the L\'evy measures is merely H\"older continuous, thereby illustrating the sharpness of our continuity assumptions.
\end{abstract}

\medskip

\medskip


\section{Introduction}\label{sec:intro}

In this paper, we investigate the regularity of viscosity solutions to degenerate elliptic integro--differential Hamilton--Jacobi equations with coercive superlinear Hamiltonians of the form
\begin{equation}\label{eq:nHJ}
\lambda u-\I u(x)+H(x,Du)=0
\qquad\text{in }\R^N,
\end{equation}
where $\lambda\ge0$, the Hamiltonian \( H:\R^N\times\R^N\to\R \) is continuous and coercive with superlinear growth in the gradient variable, and $\I$ is the nonlocal operator
\begin{equation}\label{eq:Levy}
\I u(x)
=
\int_{\R^N}
\Big(
u(x+z)-u(x)-\1_B(z)\,Du(x)\cdot z
\Big)\,
\nu_x(dz),
\end{equation}
associated with a family of L\'evy measures $(\nu_x)_{x\in\R^N}$ satisfying the uniform integrability condition
\[
\sup_{x\in\R^N}
\int_{\R^N}
\min(1,|z|^2)\,\nu_x(dz)
<\infty.
\]

Nonlocal Hamilton--Jacobi equations with coercive Hamiltonians arise naturally in stochastic control and differential games involving jump processes. Integro-differential operators appear as infinitesimal generators of the controlled dynamics in Courr\`ege form, or L\' evy form (see, for instance,~\cite{Applebaum09,Courrege65}). 
Existence and uniqueness for the stationary problem~\eqref{eq:nHJ} with $\lambda>0$ were recently established in~\cite{CLLT26}. The objective of the present paper is to complement this well-posedness theory by establishing regularity estimates for viscosity solutions.

Nonlocal operators, and the fractional Laplacian in particular, have attracted sustained interest from the PDE community over the last two decades, owing both to their connections with several areas of analysis and to the wide range of applications in which they arise; we refer to the survey by Di Nezza, Palatucci and Valdinoci~\cite{Hitch} for a broad overview of the fractional Laplacian and its functional-analytic properties, to V\'azquez~\cite{Vazquez12} for an overview of the associated nonlinear diffusion models, and to Ros-Oton~\cite{RosOton16survey} for a survey focused on the associated regularity theory. In the setting of viscosity solutions for non-divergence form equations, the systematic study of the ellipticity properties of such operators has led to a rich regularity theory.

The regularity theory of uniformly elliptic nonlocal equations is now well understood, {but} considerably less is known for the general class of operators considered here. Indeed, unlike the fractional Laplacian and other translation-invariant stable operators, the operator~$\I$ is generated by a family of possibly $x$-dependent L\'evy measures and is coupled with a coercive nonlinear Hamiltonian. As a consequence, many of the techniques available in the uniformly elliptic setting are no longer applicable.
Our main results, stated in Theorems~\ref{thm:reg-Hoe} and~\ref{thm:reg-Lip}, establish H\"older and Lipschitz regularity for viscosity solutions of~\eqref{eq:nHJ} under general structural assumptions on the family of L\'evy measures $(\nu_x)_{x\in\R^N}$, namely hypotheses~(L1)-(L3) below.


Regularity of solutions for equations of the form ~\eqref{eq:nHJ} is governed by two complementary mechanisms. The first is the ellipticity of the nonlocal operator, which regularizes solutions through suitable nondegeneracy properties of the L\'evy measures.
The second relies on certain structural properties of the Hamiltonian--typically superlinear coercivity with respect to the gradient terms--leading to the predominance of first-order terms. These, in turn, can generate regularity, even in the presence of diffusion.
Although each mechanism has been extensively investigated separately, their interaction remains much less understood. The main objective of this paper is to show how these two effects can be combined to obtain improved regularity estimates for viscosity solutions.

To illustrate these ideas, we consider throughout the introduction the prototype Hamiltonian
\begin{equation}\label{Htypique}
H(x,p)=b(x)|p|^m-f(x),
\end{equation}
where $m>1$ and $b,f\in C_b(\R^N)$. The corresponding model equation is
\begin{equation}\label{eq}
\lambda u-\I u+b(x)|Du|^m=f
\qquad\text{in }\R^N,
\end{equation}
where the L\'evy measures admit densities
\begin{equation}\label{eq:density-measure}
\nu_x(dz)=K(x,z)\,dz,
\qquad
K:\R^N\times\R^N\to\R_+.
\end{equation}
The remainder of the introduction is organized according to the two regularization mechanisms described above--regularity theory driven by ellipticity and that based on coercivity--before stating our two main results: the first exploiting the coercivity of the Hamiltonian, and the second combining both mechanisms.
\medskip

\subsection*{Regularization through ellipticity.}

The regularity theory of uniformly elliptic nonlocal equations has undergone remarkable developments over the last two decades. In a series of seminal works, Caffarelli and Silvestre
\cite{CS09,CS11,CS11b,Silvestre11} established H\"older, $C^{1,\alpha}$ and higher regularity estimates for viscosity solutions of fully nonlinear integro--differential equations under suitable uniform ellipticity assumptions of the form
\begin{equation}\label{eq:NL-unif-ell}
\frac{c}{|z|^{N+\sigma}}
\le
K(x,z)
\le
\frac{C}{|z|^{N+\sigma}},
\qquad
x\in\R^N,\; z\neq0,
\end{equation}
where $\sigma\in(0,2)$ and $0<c<C$. This framework includes, in particular, the fractional Laplacian and more generally translation-invariant stable operators. The theory has since been substantially extended in several directions, notably by Ros-Oton and Serra \cite{RS14,RS16}, who obtained sharp boundary regularity and higher regularity estimates for nonlocal elliptic equations.

A different approach, initiated by Ishii and Lions for second-order equations \cite{IL90}, exploits the viscosity solution framework to derive regularity under \emph{weak ellipticity assumptions} (see~(L4)), which in particular allows $c=0$ in \eqref{eq:NL-unif-ell}. Barles, Chasseigne and Imbert extended this method to integro--differential equations in \cite{BCI11}, proving H\"older regularity for bounded viscosity solutions of weakly elliptic equations of order $\sigma\in(0,2)$. Their results apply, in particular, to Hamilton--Jacobi equations with Hamiltonians whose growth is compatible with the order of the diffusion, typically when $m\le \sigma$. This approach was further developed by Barles, Chasseigne, Ciomaga and Imbert \cite{BCCI12}, who obtained Lipschitz regularity in the supercritical regime $\sigma\in(1,2)$ under the assumptions $b\in C^{0,\tau}$ for some $\tau>0$ and $m\le \sigma+\tau$.  The critical case $\sigma=1$ was later treated by Ciomaga, Ghilli and Topp \cite{CGT22}, who established Lipschitz regularity assuming in addition that $f\in C^{0,\tau}$. 

The common feature of all these results is that the regularization mechanism is provided by the nonlocal diffusion
and the coercivity of the gradient term plays essentially no role in the regularity argument.
The main limitation of the method lies in the fact that the growth of the Hamiltonian with respect to the gradient
must be limited--typically subquadratic.

\medskip

\subsection*{Regularity resulting from the predominance of Hamiltonian term.}

A different line of research exploits coercivity-like properties of the Hamiltonian rather than the ellipticity of the nonlocal operator. The underlying principle is that a sufficiently strong superlinear gradient term may itself regularize solutions, even in the presence of degenerate diffusion. 
This is well-known for first-order Hamilton--Jacobi equations
and this also occurs in presence of diffusion when the growth of the Hamiltonian dominates the latter, namely
when
\begin{equation}\label{bpos}
\text{for all $x\in \R^N$, \  $b(x)\ge b_0>0$, \ with $m$ large enough, \ in~\eqref{Htypique}.}
\end{equation}
H\"older regularity results were obtained in this way for viscosity solutions of second-order
nonlinear degenerate elliptic equations of the form
\begin{equation}\label{eq2}
\lambda u-\mathrm{Tr}(A(x)D^2u)+b(x)|Du|^m=f,
\end{equation}
where $A(x)=\Sigma(x)\Sigma(x)^t \geq 0$,
by Capuzzo Dolcetta, Leoni and Porretta~\cite{CDLP10} for $m>2$.
Barles, Koike, Ley and Topp \cite{BKLT15} (see also \cite{BT16b}) established the same kind of estimates
for nonlocal equations of the type~\eqref{eq} when $m>\sigma$.
A remarkable feature of their approach is that it does not require any regularity assumptions on either the coefficient $b$ or the source term $f$. 
More precisely, it has been established in~\cite{BKLT15} the following regularity for subsolutions: Lipschitz continuity when $\sigma<1$, $C^{0,\theta}$ regularity for every $\theta\in(0,1)$ when $\sigma=1$, and $C^{0,\theta}$ with
\begin{equation}\label{eq:theta0}
\theta_0 =\frac{m-\sigma}{m-1}, \quad \text{ when } \sigma\in(1,2).
\end{equation}
Related regularity results were obtained by Cardaliaguet and Rainer \cite{CR11} for L\'evy--It\^o integro-differential Hamilton--Jacobi equations with superquadratic Hamiltonians ($m\ge2$), using probabilistic methods.

Another methods that allow for the systematic exploitation of the structural properties of the Hamiltonian are
modifications of the classical \emph{Bernstein method}, see for instance~\cite{gt01, lions82, Lions85, LL89} and
Barles~\cite{Barles91, Barles21} who introduced the \emph{weak Bernstein method}, which relies solely on the viscosity
solution framework through the doubling-of-variables technique. This approach makes it possible to derive gradient
estimates under minimal regularity assumptions on the solution. It was subsequently developed in~\cite{CDLP10},
where this method made it possible to improve the H\"older estimates for~\eqref{eq2}
into Lipschitz estimates, without assuming uniform ellipticity.

This strategy was adapted by Barles, Ley and Topp~\cite{BLT17} to nonlocal Hamilton--Jacobi equations
involving L\'evy--It\^o operators of the form
\begin{equation}\label{LI}
\I u(x)
=
\int_{\R^N}
\big(
u(x+j(x,z))-u(x)-\mathbf1_B(z)Du(x)\cdot j(x,z)
\big)\,
\nu(dz),
\end{equation}
where $j:\R^N\times\R^N\to\R^N$ is a jump function and $\nu$ is a L\'evy measure. The L\'evy--It\^o structure is particularly well suited to the doubling-of-variables argument and allows the weak Bernstein method to be combined with nonlocal degenerate ellipticity in order to obtain global Lipschitz estimates. 

In these cases, the regularization mechanism is provided by coercivity-like properties of the Hamiltonian
and the estimates remain valid even for highly degenerate nonlocal operators.
\medskip

\subsection*{Our contribution.}

The present paper extends the previous contributions to the general class of integro--differential operators of the form~\eqref{eq:Levy}, beyond the L\'evy--It\^o framework considered in~\cite{BLT17}. This considerably enlarges the class of admissible nonlocal operators and allows us to combine coercivity and nonlocal ellipticity in a unified viscosity solution argument.


Our first main result, Theorem~\ref{thm:reg-Hoe}, establishes H\"older regularity for viscosity solutions of~\eqref{eq:nHJ} without assuming any uniform or weak ellipticity of the operator $\I$.  In contrast with the coercivity-based theory of~\cite{BKLT15}, our result does not require the dominance condition \( m>\sigma. \)
Instead, it applies to general superlinear Hamiltonians with
\(
m>1,
\)
thereby covering, in particular, the delicate regime \( 1<m<\sigma<2, \)
which, to the best of our knowledge, has not previously been treated in this level of generality.

The proof relies on an exponential change of variables, following the strategy introduced in~\cite{BLT17}, which enhances the coercive effect of the Hamiltonian while preserving the degenerate elliptic structure of the equation. A crucial ingredient is that we work with viscosity \emph{solutions}, rather than merely subsolutions, allowing us to exploit the equation in both directions. This additional information compensates for the absence of a dominance assumption and leads to stronger regularity estimates. In particular, this yields the improved exponent 
\begin{equation}\label{eq:theta}
\theta=\frac{m-\sigma+1}{m},
\end{equation}
which satisfies
\(
\theta>\theta_0,
\)
where $\theta_0$ denotes the exponent~\eqref{eq:theta0} obtained in~\cite{BKLT15} for bounded viscosity subsolutions. Thus, passing from subsolutions to solutions yields a genuine quantitative improvement of the regularity exponent.

Whether the exponent~\eqref{eq:theta} is optimal for bounded viscosity solutions remains an interesting open problem. The main obstacle is that, in the absence of weak ellipticity, the singularity of the kernel may still compete with the coercive gradient term in the doubling-of-variables argument. This is in sharp contrast with the local theory, where the Ishii--Jensen matrix inequality provides a fundamental tool for handling $x$-dependent diffusions (see, e.g., \cite{Jensen88,Ishii89,CIL92}).


Our second main result, Theorem~\ref{thm:reg-Lip}, exploits the interplay between the two regularization mechanisms described above. Under the additional weak ellipticity assumption  {\rm (L4)} on the operator  $\I$, which requires the nondegeneracy of its kernel on conical sectors aligned with the gradient, we combine the H\"older regularity obtained in  Theorem~\ref{thm:reg-Hoe} with an Ishii--Lions argument to establish global Lipschitz regularity.
The proof follows the spirit of the second-order strategy developed in~\cite{LN16} for equations of the form~\eqref{eq2}. The coercivity of the Hamiltonian first yields H\"older regularity through  Theorem~\ref{thm:reg-Hoe}. This preliminary estimate is then used to weaken the effective contribution of the nonlinear term. As a consequence, the weak ellipticity of the nonlocal operator becomes sufficiently strong, relative to the nonlinearity, to allow for a second application of the Ishii--Lions method. This yields the desired Lipschitz estimate.
Theorem~\ref{thm:reg-Lip} thus extends the Lipschitz regularity results of~\cite{BCCI12,BLT17}
to the general class of nonlocal operators considered in this paper.


Finally, we show that the continuity assumption {\rm (L3)} is essentially sharp. In Section~\ref{appendix} we construct an equation of the form~\eqref{eq:nHJ} whose viscosity solution is H\"older continuous but not Lipschitz once the spatial dependence \( x\longmapsto \nu_x\) is weakened from Lipschitz to merely H\"older continuous, that is, when {\rm (L3)} is replaced by the weaker assumption {\rm (L3')}. Although this does not contradict Theorem~\ref{thm:reg-Lip}, it shows that the Lipschitz continuity required in {\rm (L3)} cannot, in general, be relaxed while preserving Lipschitz regularity.
\medskip

\subsection*{Organization of the paper.}

The paper is organized as follows. In Section~\ref{sec:assumptions}, we introduce the assumptions on the Hamiltonian and the family of L\'evy measures, and state the main results. Section~\ref{sec:Hoe-reg} is devoted to the proof of the H\"older regularity estimates, while Section~\ref{sec:pf-lip} establishes Lipschitz regularity under the additional weak ellipticity assumption. Finally, Section~\ref{appendix} presents a counterexample illustrating the sharpness of assumption {\rm (L3)} by exhibiting a H\"older continuous, but non-Lipschitz, viscosity solution when the spatial regularity of the L\'evy measures is weakened.

\section{Assumptions and main results}\label{sec:assumptions}

\subsection{Notations}

For $x,y\in\R^N$, we denote by $x\cdot y$ their Euclidean scalar product and by $|x|$ the associated norm. Given $x\in\R^N$ and $r>0$, $B_r(x)$ denotes the open ball centered at $x$ with radius $r$; when $x=0$, we simply write $B_r$. We further denote by $B$ the unit ball, by $B^*:=B\setminus\{0\}$ the punctured unit ball, and for any subset $A\subset\R^N$, by $A^c$ its complement.
For a set $A\subset\R^N$, we denote by ${\rm USC}_b(A)$, ${\rm LSC}_b(A)$, $C_b(A)$, and ${\rm BUC}(A)$ the spaces of bounded upper semicontinuous, lower semicontinuous, continuous, and uniformly continuous functions on $A$, respectively. For $k\ge1$, $C^k(A)$ denotes the space of $k$--times continuously differentiable functions on $A$, and for $\theta\in(0,1]$, $C^{0,\theta}(A)$ denotes the space of $\theta$--H\"older continuous functions on $A$.
If $u:\R^N\to\R$ is smooth, we denote by $Du$ its gradient and by $D^2u$ its Hessian matrix. We write $\mathbb S^N$ for the space of $N\times N$ real symmetric matrices and denote by $I$ the $N\times N$ identity matrix.

For a measurable set $A\subset\R^N$, we define
\begin{equation}\label{eq:Levy-A}
\I[A](x,u) :=\int_A\Bigl(u(x+z)-u(x)-\1_B(z)\,Du(x)\cdot z\Bigr)\,\nu_x(dz),
\end{equation}
and  for $p\in\R^N$, we write
\begin{equation}\label{eq:Levy-A-p}
\I[A](x,u,p)
:=\int_A\Bigl(u(x+z)-u(x)-\1_B(z)\,p\cdot z\Bigr)\,\nu_x(dz),
\end{equation}
whenever the integrals are well defined. Notice that $\I[A](x,u,Du(x)) = \I[A](x, u)$.

A function $\omega:\R^+\to\R^+$ is called a modulus of continuity if $\omega(s)\to0$ as $s\to0$. We use the standard notation $o_\varepsilon(1)$ to denote a quantity $q(\varepsilon)\to0$ as $\varepsilon\searrow0$. For fixed parameters $\gamma_1,\dots,\gamma_k$, we write \( o_\varepsilon^{\gamma_1,\dots,\gamma_k}(1) \) to indicate a quantity $q^{\gamma_1,\dots,\gamma_k}(\varepsilon)\to0$ as $\varepsilon\searrow0$.

\subsection{Assumptions}\label{subsec:assumptions}

We make precise below the assumptions on the family of L\'evy measures and on the
Hamiltonian that are used throughout the paper.

\paragraph{\bf Assumptions on the L\'evy measures.}

Let $\bigl(\nu_x\bigr)_{x\in\R^N} \subset \mathcal{M}(\R^N)$ be a family of L\'evy measures satisfying the following conditions:
\begin{itemize}
\item[(L1)] There exists a constant $C_\nu>0$ such that
	\[ \sup_{x\in\R^N} \int_{\R^N} \min\bigl(1,|z|^2\bigr)\,\nu_x(dz) \le C_\nu . \]

\item[(L2)] For all $R\ge1$,
	\[ \sup_{x\in\R^N} \int_{\R^N\setminus B_R} \nu_x(dz) \le o_{1/R}(1).\]

\item[(L3)] There exist a constant $C_\nu>0$ and a nonlocal order $\sigma\in(0,2)$ such that, for all $\delta\in(0,1)$ and all $x,y\in\R^N$,
	\[ \begin{aligned}
		(i)\quad &
		\int_{B_\delta} |z|^2\,|\nu_x-\nu_y|(dz) \le C_\nu |x-y|\,\delta^{2-\sigma}, \\
		(ii)\quad &
	\int_{B\setminus B_\delta} |z|\,|\nu_x-\nu_y|(dz) \le C_\nu |x-y|
		\zeta_\sigma(\delta)
		\qquad \text{ with }
		\zeta_\sigma(\delta)=
		\begin{cases}
			\delta^{1-\sigma}, & \text{if } \sigma\neq1,\\
			|\log\delta|, & \text{if } \sigma=1,
		\end{cases}  \\
	(iii)\quad &
	\int_{{B_\delta^c}} |\nu_x-\nu_y|(dz) \le C_\nu |x-y|{\delta^{-\sigma}}.
	\end{aligned}
	\]
\end{itemize}

Assumption~(L1) is the classical L\'evy condition. Assumption~(L2) is required to control the behavior of the nonlocal terms at infinity when using localization arguments (see, e.g., \cite[Appendix~A]{CLLT26}). Assumption~(L3) is a \emph{quantified regularity assumption} on the family $\bigl(\nu_x\bigr)_{x\in\R^N}$, both near the singularity and away from it, and plays a crucial role in the estimates of the nonlocal terms. This is in contrast with the comparison result in \cite{CLLT26}, where a mere modulus of continuity was sufficient. See Remark~\ref{rk:assL2} for further discussion.

\smallskip
\paragraph{\bf Assumptions on the Hamiltonian.}
We assume that the Hamiltonian is Lipschitz continuous with respect to the space variable and satisfies a \emph{superlinear growth condition in the gradient variable} (i.e.\ $m>1$). This coercivity property is fundamental, as it ensures that the Hamiltonian dominates all other terms in the equation.
Let $H:\R^N\times\R^N\to\R$ satisfy the following:
\begin{itemize}
\item[(H0)] $H \in C(\R^N\times\R^N)$ and it satisfies 
		\[ \displaystyle M_H = \mathop{\rm sup}_{x\in\R^N}|H(x,0)| < +\infty.\]
\item[(H1)] There exist $m>1$, a constant $C_H>0$, and a modulus of continuity 
	$\zeta_H:\R_+\to\R_+$ such that, for all $x,y,p,q\in\R^N$ with $|q|\le1$,
	\[ H(y,p+q)-H(x,p) \le C_H |x-y|\bigl(1+|p|^m\bigr) + \zeta_H(|q|)\bigl(1+|p|^{m-1}\bigr).\]

\item[(H2)] There exist $m>1$, constants $b_m,b_0>0$, and $r_0>0$ such that,
for all $\mu\in(0,1]$ and all $x,p\in\R^N$ with $|p|\ge r_0$,
\[
\mu H\bigl(x,\mu^{-1}p\bigr)-H(x,p)
\ge (1-\mu)\bigl(b_m|p|^m-b_0\bigr).
\]
\end{itemize}

Assumption~(H1) is a classical continuity condition on the data, while Assumption~(H2) expresses the superlinear coercivity of the Hamiltonian $H$. This property yields the following estimate, proved in \cite{BCT20}: if (H1)--(H2) hold with $m>1$, then there exist constants $C,r_0>0$ such that, for all $x,p\in\R^N$ with $|p|\ge r_0$,
\[ H(x,p)\ge \frac{1}{C}|p|^m - C . \]

Hamiltonians of the form~\eqref{Htypique} under assumption~\eqref{bpos} with $m>1$,
and $b,f$ bounded and Lipschitz continuous
satisfy~(H1)--(H2).

\subsection{Main results}\label{subsec:main}

The main contribution of this paper is the proof of H\"older regularity estimates for degenerate nonlocal equations with superlinear Hamiltonians. We begin by recalling the notion of viscosity solutions for equation~\eqref{eq:nHJ}; see \cite{BI08} for equivalent definitions.

\begin{definition}[Viscosity solution]\label{def:visc-sol} 
$\;$

\begin{itemize}
\item[(i)] A function $u\in {\rm USC}_b(\R^N)$ is a \emph{viscosity subsolution} 
	of~\eqref{eq:nHJ} if, for any $0<\delta<1$, any $\bar x\in\R^N$, and any test
	function $\phi\in C^2(B_\delta(\bar x))\cap L^\infty(\R^N)$ such that $\bar x$
	is a local maximum of $u-\phi$ in $B_\delta(\bar x)$, one has
	\[ \lambda u(\bar x) - \I[B_\delta](\bar x,\phi) - \I[B^c_\delta](\bar x,u,D\phi(\bar x))
		+ H(\bar x,D\phi(\bar x))\le 0.\]

\item[(ii)] A function $u\in {\rm LSC}_b(\R^N)$ is a \emph{viscosity supersolution}
	of~\eqref{eq:nHJ} if, for any $0<\delta<1$, any $\bar x\in\R^N$, and any test
	function $\phi\in C^2(B_\delta(\bar x))\cap L^\infty(\R^N)$ such that $\bar x$
	is a local minimum of $u-\phi$ in $B_\delta(\bar x)$, one has
	\[ \lambda u(\bar x) - \I[B_\delta](\bar x,\phi) - \I[B^c_\delta](\bar x,u,D\phi(\bar x))
		+ H(\bar x,D\phi(\bar x))	\ge 0.\]

\item[(iii)] A function $u\in C_b(\R^N)$ is a \emph{viscosity solution} of~\eqref{eq:nHJ} 
	if it is both a viscosity subsolution and a viscosity supersolution.
\end{itemize}
\end{definition}

The following theorem describes how the H\"older regularity of solutions to~\eqref{eq:nHJ} depends on the nonlocal order $\sigma\in(0,2)$. See Section~\ref{sec:Hoe-reg} for the proof.

\begin{theorem}[H\"older regularity]\label{thm:reg-Hoe}
Let $\lambda\ge0$. Assume that the family of L\'evy measures $\bigl(\nu_x\bigr)_{x\in\R^N}$ satisfies assumptions (L1)--(L3) with nonlocal order $\sigma\in(0,2)$, and that the Hamiltonian $H$ satisfies assumptions (H1)--(H2) with $m>1$. Let $u\in C_b(\R^N)$ be a viscosity solution of~\eqref{eq:nHJ}. Then the following assertions hold:
\begin{itemize}
\item[(i)] If $\sigma\in(0,1)$, then $u$ is Lipschitz continuous.
\item[(ii)] If $\sigma=1$, then $u$ is H\"older continuous for every exponent $\theta \in (0, 1)$.
\item[(iii)] If $\sigma\in(1,2)$, then $u$ is H\"older continuous with exponent
		$\displaystyle \theta=\frac{m-\sigma+1}{m}.$
\end{itemize}
Moreover, the corresponding seminorm of $u$ depends only on $\osc(u)$ and the given data.
\end{theorem}

\begin{remark}\label{rk:assL2}\normalfont
Our regularity result improves, in certain directions, the estimates obtained in~\cite{BKLT15} (see also~\cite{BT16} for the case $\sigma\in(0,1)$ and $m>\sigma$). More precisely, in ~\cite{BKLT15}  the family of L\'evy measures $\bigl(\nu_x\bigr)_{x\in\R^N}$ satisfies the following assumptions:
\begin{itemize}
\item[(M1)] For all $R>0$ and $\alpha\in[0,2]$, there exists $C_R>0$ such that, for all $\delta>0$,
	\[ \sup_{x\in\bar B_R} \int_{B_\delta^c} \min\{1,|z|^\alpha\}\,\nu_x(dz)	
		\le C_R\, h_{\alpha,\sigma}(\delta),
	\quad \text{ with } \quad 
	h_{\alpha,\sigma}(\delta)=
	\begin{cases}
		\delta^{\alpha-\sigma}, & \text{if } \alpha<\sigma,\\
		|\ln\delta|+1, & \text{if } \alpha=\sigma,\\
		1, & \text{if } \alpha>\sigma.
	\end{cases}\]

\item[(M2)] For all $R>0$ and $\alpha\in(\sigma,2]$, there exists $C_R>0$ such that, for all $\delta\in(0,1)$,
	\[ \sup_{x\in\bar B_R} \int_{B_\delta} |z|^\alpha\,\nu_x(dz)
		\le C_R\,\delta^{\alpha-\sigma}.\]
\end{itemize}
On the other hand, the Hamiltonian $H$ satisfies (H1)--(H2) with $m>\max\{1,\sigma\}$. Then, by~\cite[Theorem~2.1]{BKLT15}, any viscosity \emph{subsolution} of~\eqref{eq:nHJ} is \emph{locally} H\"older continuous, that is 
\begin{itemize}
\item[(i)] If $\sigma\in(0,1)$, then $u$ is locally Lipschitz continuous.
\item[(ii)] If $\sigma=1$, then $u$ is locally H\"older continuous for all exponents $\theta\in(0,1)$.
\item[(iii)] If $\sigma\in(1,2)$, then $u$ is locally H\"older continuous with exponent
	$ \displaystyle \theta_0=\frac{m-\sigma}{m-1}.$
\end{itemize}
In Theorem~\ref{thm:reg-Hoe}, by considering \emph{solutions} rather than merely subsolutions, we are able to improve the H\"older exponent in the regime $\sigma\in(1,2)$, since
\[ \theta=\frac{m-\sigma+1}{m}> \theta_0 = \frac{m-\sigma}{m-1}. \]
On the other hand, while our result requires the regularity assumption (L3) on the dependence of the family of L\'evy measures $\bigl(\nu_x\bigr)_{x\in\R^N}$ with respect to the spatial variable $x$, it relies on weaker assumptions on the L\'evy measures themselves when
comparing~(L1) with~(M1)--(M2). 
This difference stems from the fact that the two results are obtained by fundamentally different approaches. The method in~\cite{BKLT15}, based on the use of subsolutions and barrier functions, requires precise upper bounds on the L\'evy measure. In contrast, our approach relies on a nonlinear change of variables and on the comparison of the nonlocal terms at different spatial points, which explains the need for assumption~(L3).
Finally, we note that Assumption~(L2) is used to obtain global H\"older estimates in the present work, and a similar condition is also required in~\cite{BKLT15} in order to extend the regularity results to unbounded domains.

We do not know whether this new exponent is optimal for $\sigma\in[1,2)$. We point out that Lipschitz continuity for all $\sigma\in(0,2)$ is established in~\cite{BLT17} only for nonlocal operators in L\'evy--It\^o form, and not for general L\'evy measures.
A similar Lipschitz regularity result is obtained in~\cite{BCCI12} under a weak ellipticity assumption on the nonlocal operator. Within this framework, we are able to extend the H\"older regularity result of Theorem~\ref{thm:reg-Hoe} to Lipschitz regularity, see Theorem~\ref{thm:reg-Lip} below. 
\end{remark}

Based on the regularity results established in Theorem~\ref{thm:reg-Hoe}, we show that viscosity solutions of~\eqref{eq:nHJ} are in fact Lipschitz continuous when the associated nonlocal integro-differential operators are \emph{weakly elliptic}. More precisely, we assume that the family of L\'evy measures $\bigl(\nu_x\bigr)_{x\in\R^N}$ additionally satisfies the following cone-ellipticity condition.

\begin{itemize}
\item[(L4)] There exist $\sigma\in(0,2)$, $C_\nu>0$ and $\widehat\eta \in (0, 1)$ such that, for every $p\in\R^N \setminus \{ 0 \}$, $x\in\R^N$, $\delta \in (0, 1)$ and $\eta \in (0, \widehat\eta)$,
\[
\int_{\mathcal C^{\eta}_{\delta}(p)} |z|^2\,\nu_x(dz)
\ge C_\nu\,\eta^{\frac{N-1}{2}}\,\delta^{2-\sigma},
\]
where
\[
\mathcal C^{\eta}_{\delta}(p)
:=\bigl\{ z\in B_\delta : (1-\eta)|z||p|\le |p\cdot z|\bigr\}.
\]
\end{itemize}

The following result complements Theorem~\ref{thm:reg-Hoe} in the regime $\sigma\ge1$. See Section~\ref{sec:pf-lip} for the proof.

\begin{theorem}[Lipschitz regularity]\label{thm:reg-Lip}
Let $\lambda\ge0$. Assume that the family of L\'evy measures $\bigl(\nu_x\bigr)_{x\in\R^N}$ satisfies assumptions~(L1)--(L4) with $\sigma\in(0,2)$, and that the Hamiltonian $H$ satisfies assumptions~(H1)--(H2) with $m>1$. Then any bounded viscosity solution $u$ of~\eqref{eq:nHJ} is Lipschitz continuous. Moreover, the Lipschitz constant of $u$ depends only on ${\osc(u)}$ and the given data.
\end{theorem}

\section{Proof of H\"older regularity - Theorem~\ref{thm:reg-Hoe}} 
\label{sec:Hoe-reg}

This section is devoted to the proof of Theorem~\ref{thm:reg-Hoe}, following the strategy of~\cite{BLT17}: an exponential change of variables combined with the doubling-of-variables method, a classical method well suited to Hamilton--Jacobi  equations with superlinear Hamiltonians.

\begin{proof}[Proof of Theorem~\ref{thm:reg-Hoe}]
First, we normalize $u$ by replacing it with $\tilde{u}= u - \inf_{\R^N} u + 1.$ The new function still solves an equation of the form \eqref{eq:nHJ}, with $H$ replaced by $ H(x,p) + \lambda(\inf_{\R^N} u - 1). $ All the standing assumptions are preserved after this modification. In particular, by the comparison results in~\cite{CLLT26}, the quantity  $\lambda \|u\|_\infty$ remains bounded uniformly in  $\lambda > 0$. Notice that $\tilde{u} \geq 1$ in $\R^N$ and 
\[ \|\tilde{u}\|_\infty\leq \osc(u)+1. \]

We then introduce the change of variables $\tilde{u} = e^v$, so that $v \ge 0$ in $\R^N$.  By standard viscosity solution arguments, $v$ satisfies
\begin{equation}\label{eq:newHJ}
	- \J(x,v) + \widetilde{H}(x,v,Dv) = 0 \qquad \text{in } \R^N,
\end{equation}
where the modified Hamiltonian is defined by
\[
\widetilde{H}(x,r,p) := e^{-r} H\bigl(x, e^r p\bigr) + \lambda.
\]
and the associated (nonlinear) nonlocal operator $\J$ is defined by 
\begin{equation}\label{J-oper}
\J(x,v) := \int_{\R^N} \Bigl(e^{v(x+z)-v(x)} - 1 - \mathbf{1}_B(z)\,Dv(x)\cdot z\Bigr)\,\nu_x(dz).
\end{equation}
As for the original nonlocal operator $\I$, we shall use analogous notation for $\J$,  depending on the domain of integration, as in \eqref{eq:Levy-A}--\eqref{eq:Levy-A-p}.

We now show that $v$ is $\theta$--H\"older continuous with a seminorm $L$ depending on $\osc(v)\leq \osc(u)$. This immediately implies that $u$ is H\"older continuous with the same exponent and a seminorm $(\osc(u)+1)L$, which gives the desired result.

Fix $0 < \theta \le 1$ and argue by contradiction.  Assume that $v$ is not $\theta$--H\"older continuous. Then, for all $L > 1$ large enough, there exists $\varepsilon_L > 0$ such that
\begin{equation}\label{eq:holder-contradiction}
	\sup_{(x,y)\in \R^N \times \R^N} \Bigl\{ v(x) - v(y) - L |x-y|^\theta \Bigr\} \ge 4 \varepsilon_L.
\end{equation}

Let $\psi: \R^N \to \R$ be a nonnegative smooth function such that
\[
	\psi = 0 \quad \text{in } B_1,  \qquad
	\psi = \osc(v) + 1 \quad \text{in } B_2^c,
\]
and define $\psi_\beta(x) := \psi(\beta x)$ with $\beta \in (0,1)$, which will be sent to zero. Then, for all $\beta > 0$ sufficiently small, we have
\begin{equation}\label{eq:reg-sup}
\max_{(x,y)\in \R^N \times \R^N}
\Bigl\{ v(x) - v(y) - L|x-y|^\theta - \psi_\beta(x) \Bigr\}
\ge 3\varepsilon_L.
\end{equation}
We denote by $(\bar x, \bar y)$ a point where the maximum is attained. Note that $\bar x \neq \bar y$, since otherwise the expression in \eqref{eq:reg-sup} would be nonpositive. For later use, we note that, by the boundedness and uniform continuity of $v$ in $\R^N$ (with modulus $\omega$),
\begin{equation}\label{eq:L-theta}
	L\;|\bar x-\bar y|^\theta \le \osc(v) \le \osc(u),
\end{equation}
The first inequality follows from the boundedness of $v$, while the second one follows from the definition of the modulus of continuity and the fact that $\psi_\beta\geq 0$.  In particular, $|\bar x - \bar y| \to 0$ as $L \to \infty$, while it remains bounded away from zero independently of $\beta$.

Therefore, denoting
\begin{eqnarray}\label{def-phi}
	\phi(x,y):=L|x-y|^\theta,
\end{eqnarray}
the functions
\begin{align*}
	x \mapsto & v(\bar y)+\phi(x,\bar y)+\psi_\beta(x),\\
	y \mapsto & v(\bar x)-\phi(\bar x,y)
\end{align*}
are smooth test functions touching $v$ from above at $\bar x$ and from below at $\bar y$, respectively.
The viscosity inequalities for $v$ then read, for any $0<\delta<1$,
\begin{align*}
	- \J[B_\delta]\bigl(\bar x,\phi(\cdot,\bar y)+\psi_\beta\bigr)
	- \J[B\setminus B_\delta]\bigl(\bar x,v,p+q\bigr)
	- \J[B^c]\bigl(\bar x,v\bigr)
	+ \widetilde H\bigl(\bar x,v(\bar x),p+q\bigr)
&\le 0,\\
	- \J[B_\delta]\bigl(\bar y,-\phi(\bar x,\cdot)\bigr)
	- \J[B\setminus B_\delta]\bigl(\bar y,v,p\bigr)
	- \J[B^c]\bigl(\bar y,v\bigr)
	+ \widetilde H\bigl(\bar y,v(\bar y),p\bigr)
&\ge 0,
\end{align*}
where
\begin{align*}
	p &  := D_x\phi(\bar x,\bar y) = -D_y\phi(\bar x,\bar y)
		= \theta L(\bar x-\bar y)|\bar x-\bar y|^{\theta-2}, \\
	q & := D\psi_\beta(\bar x).
\end{align*}

Subtracting the two inequalities yields
\begin{equation}\label{eq:visc-ineq}
\begin{split}
	\widetilde H\bigl(\bar x,v(\bar x),p+q\bigr)
	-\widetilde H\bigl(\bar y,v(\bar y),p\bigr)
	\le\;& 
	\J[B_\delta]\bigl(\bar x,\phi(\cdot,\bar y)+\psi_\beta\bigr)
	-\J[B_\delta]\bigl(\bar y,-\phi(\bar x,\cdot)\bigr)	\\
	&
	+ \J[B\setminus B_\delta]\bigl(\bar x,v,p+q\bigr)
	-\J[B\setminus B_\delta]\bigl(\bar y,v,p\bigr)	\\
	&
	+ \J[B^c]\bigl(\bar x,v\bigr)
	-\J[B^c]\bigl(\bar y,v\bigr).
\end{split}
\end{equation}

We now estimate each term in~\eqref{eq:visc-ineq} in order to derive a contradiction.

\begin{lemma}[Estimate of the Hamiltonian difference]\label{lem:Hest}
Let $\lambda \ge 0$, and assume that the Hamiltonian $H$ satisfies
assumptions~(H1)--(H2).
Then there exists a constant $C_0>0$, depending on $m$, $\theta \in (0,1]$, and $\osc(v)$, such that, for all $L$ large enough,
\begin{equation}\label{eq:Hest}
	\widetilde H\bigl(\bar x,v(\bar x),p+q\bigr)
	-\widetilde H\bigl(\bar y,v(\bar y),p\bigr)
	\ge
	C_0\,L^{m+1}\,|\bar x-\bar y|^{\theta+m(\theta-1)}
	- o_\beta(1)\, |\bar x - \bar y|^{-(m-1)}.
\end{equation}
Here $o_\beta(1) \to 0$ as $\beta \to 0$, uniformly with respect to $L$, independent of $\osc(v)$).
\end{lemma}

\begin{proof}[Proof of Lemma~\ref{lem:Hest}] 
We sketch the proof, following \cite[Proof of Theorem~3.1]{BLT17}. We decompose
\begin{align*}
	&\widetilde H\bigl(\bar x,v(\bar x),p+q\bigr)
	-\widetilde H\bigl(\bar y,v(\bar y),p\bigr) \\
	&\qquad =
	\Bigl[
		\widetilde H\bigl(\bar x,v(\bar x),p+q\bigr)
		-\widetilde H\bigl(\bar x,v(\bar x),p\bigr)
	\Bigr]
	+
	\Bigl[
		\widetilde H\bigl(\bar x,v(\bar x),p\bigr)
		-\widetilde H\bigl(\bar y,v(\bar y),p\bigr)
	\Bigr].
\end{align*}

We first estimate the second bracket. Let
\[
	\mu:=e^{v(\bar y)-v(\bar x)}\in(0,1].
\]
Using the definition of $\widetilde H$, and assumption~(H2), we have
\begin{align*}
	\widetilde H\bigl(\bar x,v(\bar x),p\bigr)
	-\widetilde H\bigl(\bar y,v(\bar y),p\bigr) 
	&=e^{-v(\bar y)}
	\Bigl[
		\mu H\bigl(\bar x,\mu^{-1}e^{v(\bar y)}p\bigr)
		- H\bigl(\bar y,e^{v(\bar y)}p\bigr)
	\Bigr]. \\
	&\ge e^{-v(\bar y)}
	\Bigl[
		(1-\mu)\bigl(b_m|e^{v(\bar y)}p|^m-b_0\bigr)
		-C_H|\bar x-\bar y|\bigl(1+|e^{v(\bar y)}p|^m\bigr)
	\Bigr].
\end{align*}
Next, by \eqref{eq:reg-sup} and the fact that $\psi_\beta\ge 0$,
\[
	1-\mu =1-e^{v(\bar y)-v(\bar x)}
	\ge e^{-\osc(v)}\bigl(v(\bar x)-v(\bar y)\bigr)
	\ge e^{-\osc(v)}L|\bar x-\bar y|^\theta.
\]
Recalling that
$ |p|=\theta L|\bar x-\bar y|^{\theta-1}, $
we deduce
\begin{align*}
	  \widetilde H\bigl(\bar x,v(\bar x),p\bigr)
	- \widetilde H\bigl(\bar y,v(\bar y),p\bigr)
	\ge\; &
	\theta^m L^{m+1}|\bar x-\bar y|^{\theta+m(\theta-1)}
	\Bigl[
	b_m e^{-\osc(v)} - C_H L^{-1}|\bar x-\bar y|^{1-\theta}
	\Bigr]\\
	& - e^{-v(\bar y)}\Bigl(b_0(1-\mu)+C_H|\bar x-\bar y|\Bigr)
\end{align*}
Since $v\ge 0$ and $|\bar x-\bar y|\to0$ as $L\to\infty$, it follows that, for $L$ large enough,
\begin{equation}\label{ineq:H-p}
	\widetilde H\bigl(\bar x,v(\bar x),p\bigr)
	-\widetilde H\bigl(\bar y,v(\bar y),p\bigr)
	\ge
	2C_0\,L^{m+1}|\bar x-\bar y|^{\theta+m(\theta-1)}
\end{equation}
for some constant $C_0>0$ depending only on $m$, $\theta$, and $\osc(v)$.

We now estimate the first bracket. By definition of $\widetilde H$,
\begin{align*}
	\widetilde H\bigl(\bar x,v(\bar x),p+q\bigr)
	-\widetilde H\bigl(\bar x,v(\bar x),p\bigr)
	&= e^{-v(\bar x)}
	\Bigl[
		H\bigl(\bar x,e^{v(\bar x)}(p+q)\bigr)
		-H\bigl(\bar x,e^{v(\bar x)}p\bigr)
	\Bigr].
\end{align*}
Using assumption~(H1) with
$ P=e^{v(\bar x)}p, Q=e^{v(\bar x)}q, $
we obtain
\[
	\widetilde H\bigl(\bar x,v(\bar x),p+q\bigr)
	-\widetilde H\bigl(\bar x,v(\bar x),p\bigr)
	\ge -\,e^{-v(\bar x)}\zeta_H(|Q|)\,|P|^{m-1}.
\]
Since $|q|=|D\psi_\beta(\bar x)|=o_\beta(1)$, it follows that $|Q|=o_\beta(1)$, while 
$ |P| = e^{v(\bar x)}\theta L|\bar x-\bar y|^{\theta-1}. $
Using the boundedness of $v$ and \eqref{eq:L-theta}, we deduce
\begin{equation}\label{ineq:H-p+q}
	\widetilde H\bigl(\bar x,v(\bar x),p+q\bigr)
	-\widetilde H\bigl(\bar x,v(\bar x),p\bigr)
	\ge -\,o_\beta(1)\,L^{m-1}|\bar x-\bar y|^{(m-1)(\theta-1)}
	\ge -\,o_\beta(1)\,|\bar x-\bar y|^{-(m-1)}.
\end{equation}
Combining \eqref{ineq:H-p+q} with \eqref{ineq:H-p}, we conclude that \eqref{eq:Hest} holds.
\end{proof}

\noindent{\bf Estimate of the nonlocal terms.}
\noindent 
For $\varphi=v,\phi_\beta$ and $\phi$ given by~\eqref{def-phi}, for fixed $x,y\in\R^N$ and any $z\in\R^N$, we define
\begin{eqnarray*}
	\Delta_x \varphi(z) & := &\varphi(x+z) - \varphi(x),\\
	\Delta_{x} \phi(z,y) &:= &\phi(x+z,y) - \phi(x,y),  \\
	\Delta_{y} \phi(x,z) &:= &\phi(x,y+z) - \phi(x,y).
\end{eqnarray*}
In view of the concavity of the function $t\mapsto t^\theta$, together with the $C^1$ regularity and boundedness of $\psi_\beta$,
we have, for all $z\in\R^N$,
\begin{equation}\label{eq:delta-est}
\begin{aligned}
	\bigl|\Delta_{\bar x}\phi(z,\bar y)\bigr|
		&\le 	C\,L|\bar x-\bar y|^{\theta-1}|z|,	\\
	\bigl|\Delta_{\bar x}\psi_\beta(z)\bigr|
		&\le	\min\bigl(C\beta\,|z|,\;  \osc(u) +1\bigr).
\end{aligned}
\end{equation}
With this notation, using the maximality of $(\bar x,\bar y)$ in \eqref{eq:reg-sup}, we have, for all $z\in\R^N$,
\begin{equation}\label{eq:delta-max}
\begin{aligned}
\Delta_{\bar x} v(z) 
& \le 	\Delta_{\bar x} \phi(z,\bar y)	+
	\Delta_{\bar x} \psi_\beta(z)	\\ 
	\Delta_{\bar y} v(z) 
&\ge -\Delta_{\bar y} \phi(\bar x,z)\\
	\Delta_{\bar x} v(z) - \Delta_{\bar y} v(z)
& \le \Delta_{\bar x} \psi_\beta(z).
\end{aligned}
\end{equation}

\begin{lemma}[Estimate of the nonlocal difference inside $B_\delta$]
\label{lem:NLest-in}
Assume that the family of L\'evy measures $\bigl(\nu_x\bigr)_{x\in\R^N}$ satisfies~(L1). Then, for all $0<\delta<|\bar x-\bar y|/2$, there exists a constant $C>0$, depending only on the data and $\osc(u)$, such that
\begin{equation}\label{eq:NLest-in}
\begin{aligned}
	\bigl|\J[B_\delta]\bigl(\bar x,\phi(\cdot,\bar y)+\psi_\beta\bigr)\bigr|
	&\le
	C\;\Bigl(|\bar x-\bar y|^{-2}+o_\beta(1)\Bigr)\,o^{\bar x}_\delta(1),	\\
	\bigl|\J[B_\delta]\bigl(\bar y,-\phi(\bar x,\cdot)\bigr)\bigr|
	&\le 
	C\; |\bar x-\bar y|^{-2}\,o^{\bar y}_\delta(1).
\end{aligned}
\end{equation}
\end{lemma}

\begin{proof}
The nonlocal terms
$
\J[B_\delta]\bigl(\bar x, \phi(\cdot, \bar y)+\psi_\beta\bigr)
\text{ and }
\J[B_\delta]\bigl(\bar y, -\phi(\bar x, \cdot)\bigr)
$
are estimated in a similar way. We therefore only detail the estimate of the first one, which is slightly more delicate because of the presence of the localization term $\psi_\beta$. Let
\[
	\phi_\beta (x,y):=\phi(x,y)+\psi_\beta(x).
\]
The nonlocal term can be written as
\begin{align*}
	\J[B_\delta]\bigl(\bar x,\phi(\cdot,\bar y)+\psi_\beta\bigr) 
	= & 
	\int_{B_\delta} \Bigl(e^{\Delta_{\bar x}\phi_\beta(z,\bar y)} -1 -\Delta_{\bar x}\phi_\beta(z,\bar y) \Bigr)\,\nu_{\bar x}(dz) \\
	&+
	\int_{B_\delta} \Bigl( \Delta_{\bar x}\phi(z,\bar y)-p\cdot z \Bigr)\,\nu_{\bar x}(dz)\\
	&+
	\int_{B_\delta} \Bigl( \Delta_{\bar x}\psi_\beta(z)-q\cdot z \Bigr)\,\nu_{\bar x}(dz).
\end{align*}
In view of \eqref{eq:delta-est}, for $z\in B_\delta$,
\begin{equation*}
	| \Delta_{\bar x}\phi_\beta(z,\bar y)| 
	\le	C\Bigl(L|\bar x-\bar y|^{\theta-1}+\beta\Bigr)|z|.
\end{equation*}
It then follows from Taylor's theorem the following estimates, for all $z\leq  |\bar x-\bar y|/2$.
\begin{eqnarray*}
	e^{\Delta_{\bar x}\phi_\beta(z,\bar y)} -1-\Delta_{\bar x}\phi_\beta(z,\bar y)
	&\leq & 
	C e^{|\Delta_{\bar x}\phi(z,\bar y)|+ 
	|\Delta_{\bar x}\psi_\beta(z)|}\left( |\Delta_{\bar x}\phi(z,\bar y)|^2 + 
	|\Delta_{\bar x}\psi_\beta(z)|^2\right)\\
	&\leq & 
	C e^{CL|\bar x-\bar y|^{\theta-1}|z|  + \osc(u)+1}
	\left( L^2|\bar x-\bar y|^{2(\theta-1)}+\beta^2\right) |z|^2 \\
	&\leq & 
	C e^{C(\osc(u)+1)} \left( L^2|\bar x-\bar y|^{2(\theta-1)}+\beta^2\right) |z|^2,
\end{eqnarray*}
where the last inequality comes from $z\leq |\bar x-\bar y|/2$ and~\eqref{eq:L-theta}. Moreover
\begin{eqnarray} \nonumber
	\Delta_{\bar x}\psi_\beta(z)-q\cdot z
	&\le &
	\frac{1}{2} \|D^2\psi_\beta \|_\infty |z|^2 \leq C \beta^2|z|^2,
	\\ \label{eq:delta-est-order2}
	\Delta_{\bar x}\phi(z,\bar y)-p\cdot z
	&\le &
	\frac{1}{2} \sup_{\zeta\in \bar{B}(0,\frac{|\bar x-\bar y|}{2})} |D^2 \phi(\bar x+\zeta,\bar y)| |z|^2
	\leq C L|\bar x-\bar y|^{\theta-2}|z|^2, 
\end{eqnarray}
since $\zeta\to \phi(\bar x+\zeta,\bar y)$ is differentiable in $\bar{B}(0,\frac{|\bar x-\bar y|}{2})$.

Integrating the above estimates over $B_\delta$, and using~(L1) together with \eqref{eq:L-theta}, we obtain
\begin{align*}
	\bigl|\J[B_\delta]\bigl(\bar x,\phi(\cdot,\bar y)+\psi_\beta\bigr)\bigr|
	\le &
	C\Bigl( L^2|\bar x-\bar y|^{2(\theta-1)} + L|\bar x-\bar y|^{\theta-2} + o_\beta(1)\Bigr)
	\int_{B_\delta} |z|^2\,\nu_{\bar x}(dz) \\
	\le &
	C\Bigl( L^2|\bar x-\bar y|^{2(\theta-1)} + L|\bar x-\bar y|^{\theta-2} + o_\beta(1) \Bigr)\,o^{\bar x}_\delta(1) \\
	\le &
	C\Bigl( |\bar x-\bar y|^{-2} + o_\beta(1) \Bigr)\,o^{\bar x}_\delta(1),
\end{align*}
where $C$ depends on $\osc(u)$. This yields the desired estimate for $\J[B_\delta]\bigl(\bar x,\phi(\cdot,\bar y)+\psi_\beta\bigr)$.
The estimate for $\J[B_\delta]\bigl(\bar y,-\phi(\bar x,\cdot)\bigr)$ follows in the same way (and is simpler since it does not involve $\psi_\beta$), which concludes the proof.
\end{proof}

\begin{lemma}[Estimate of the nonlocal difference outside $B$]
\label{lem:NLest-out}
Assume that the family of L\'evy measures $\bigl(\nu_x\bigr)_{x\in\R^N}$ satisfies~(L2) and~(L3)(iii). Then there exists a constant $C>0$, depending only on the data and on $\osc(u)$ , such that
\begin{equation}\label{eq:NLest-out}
	\J[B^c](\bar x,v)-\J[B^c](\bar y,v) \le C\,|\bar x-\bar y|+o_\beta(1).
\end{equation}
\end{lemma}

\begin{proof}
The nonlocal difference reads
\begin{align*}
    \J[B^c](\bar x, v) - \J[B^c](\bar y, v)
    =&  \int_{B^c}\bigl(e^{\Delta_{\bar x}v(z)}-1\bigr)\nu_{\bar x}(dz)
        - \int_{B^c}\bigl(e^{\Delta_{\bar y}v(z)}-1\bigr)\nu_{\bar y}(dz) \\
    =& \int_{B^c} \Bigl( e^{\Delta_{\bar x}v(z)} - e^{\Delta_{\bar y}v(z)} \Bigr)\nu_{\bar x}(dz)
    +   \int_{B^c} \bigl(e^{\Delta_{\bar y}v(z)}-1\bigr) \bigl(\nu_{\bar x}-\nu_{\bar y}\bigr)(dz).
\end{align*}
Since $ |\Delta_{\bar x}v(z)|,\ |\Delta_{\bar y}v(z)| \le \osc(v)\leq \osc(u)$,
the mean value theorem, together with~\eqref{eq:delta-max} yields
\[
	e^{\Delta_{\bar x}v(z)}-e^{\Delta_{\bar y}v(z)}
	\le e^{\osc(u)} \bigl(\Delta_{\bar x}v(z)-\Delta_{\bar y}v(z)\bigr)
	\le e^{\osc(u)} \Delta_{\bar x}\psi_\beta(z).
\]
In view of~(L2),~(L3)(iii), and \cite[Lemma B.1]{CLLT26}, we deduce after integration that
\begin{align*}
	\J[B^c](\bar x, v) - \J[B^c](\bar y, v)
	\le &
	e^{\osc(u)}\int_{B^c} |\Delta_x \psi_\beta(z)|\,\nu_{\bar x}(dz) +
	\bigl(e^{\osc(u)}+1\bigr)\int_{B^c} |\nu_{\bar x}-\nu_{\bar y}|(dz) \\
	\le &
	e^{\osc(u)} o_\beta(1) + \bigl(e^{\osc(u)}+1\bigr)|\bar x-\bar y|.
\end{align*}
This concludes the result.
\end{proof}

\begin{lemma}[Estimate of the nonlocal difference in the crown $B\setminus B_\delta$]
\label{lem:NLest-crown}
Assume that the family of L\'evy measures $\bigl(\nu_x\bigr)_{x\in\R^N}$ satisfies~(L3)(i)--(ii). Then there exists a constant $C>0$, depending only on the data, $\osc(u)$, and  $\theta\in(0,1]$, such that, for every $ 0<\delta< |\bar x-\bar y|/2$, the following estimate holds:
\begin{equation}\label{NLest-crown}
\J[B\setminus B_\delta](\bar x,v, p+q)
-
\J[B\setminus B_\delta](\bar y,v, p)
\le
CL
|\bar x-\bar y|^{\theta}\Bigl(\zeta_\sigma(|\bar x-\bar y|)
+
\beta |\bar x-\bar y|^{-1} \Bigr)
+o_\beta(1)\zeta_\sigma(|\bar x - \bar y|).
\end{equation}
\end{lemma}

\begin{proof}
As in \cite[Proof of Theorem~3.1]{BCI11}, we use the Hahn--Jordan decomposition
of the signed measure $\nu=\nu_{\bar x}-\nu_{\bar y}$. More precisely, we write
$\nu=\nu^+-\nu^-$, where $\nu^+$ and $\nu^-$ are respectively the positive and
negative parts of $\nu$. Setting $\Theta=\mathrm{supp}\,\nu^+$, we introduce the
nonnegative measure
$ \tilde\nu := (1-\pmb{1}_\Theta)\nu_{\bar x} + \pmb{1}_\Theta \nu_{\bar y}, $
and note that
$
\nu_{\bar x}=\nu^+ + \tilde\nu, \;
\nu_{\bar y}=\nu^- + \tilde\nu.
$
We emphasize that $\tilde\nu$ still satisfies~(L1), while the measures $\nu^+$ and $\nu^-$ are controlled by the total variation measure $|\nu_{\bar x}-\nu_{\bar y}|$.

Denote, for fixed $p,x\in\R^N$
\[ 
	\ell_{x} v(z;p) := e^{\Delta_{x} v(z)} - 1 - p\cdot z. 
\]
which correspond to the truncated exponential increments. With this notation, we can write the nonlocal difference as
\begin{eqnarray*}
	&&  
	\J[B\setminus B_\delta](\bar x,v,p+q) -
	\J[B\setminus B_\delta](\bar y,v,p)\\
	&& = 
	\int_{B\setminus B_\delta} \ell_{\bar x}v(z;p+q)\,\nu_{\bar x}(dz) - 
	\int_{B\setminus B_\delta} \ell_{\bar y}v(z;p)\,\nu_{\bar y}(dz)
	= \tilde \J + \J^+  + \J^-,
\end{eqnarray*}
where
\begin{eqnarray*}
	\tilde \J &:=& \int_{B\setminus B_\delta} 
		\bigl(\ell_{\bar x}v(z;p+q)-\ell_{\bar y}v(z;p)\bigr)\,\tilde\nu(dz), \\
	\J^+ &:=& \int_{B\setminus B_\delta} \ell_{\bar x}v(z;p+q)\,\nu^+(dz), \\
	\J^- &:=& - \int_{B\setminus B_\delta} \ell_{\bar y}v(z;p)\,\nu^-(dz).
\end{eqnarray*}
We estimate each term separately. We note beforehand that, from~\eqref{eq:delta-est}-\eqref{eq:delta-max}, for all $z\in B$, the estimates hold:
\begin{equation}\label{eq:delta-estimates}
\begin{aligned}
\Delta_{\bar x}v(z) 
	&\leq  C\left(L |\bar x - \bar y|^{\theta-1} + \beta \right)|z|, \\
	\Delta_{\bar y} v(z)
	&\geq  - CL |\bar x - \bar y|^{\theta-1}|z|.
\end{aligned}
\end{equation}

\medskip
\noindent\emph{Estimate of $\tilde\J$.}
For the first term, we observe that
\[
	\tilde \J = \int_{B\setminus B_\delta} 
	\left( e^{\Delta_{\bar x}v(z)} - e^{\Delta_{\bar y}v(z)} -  D\psi_\beta(\bar x) \cdot z \right)\tilde \nu(dz).
\]
Looking on the subset $\mathcal{\tilde P}_+$, where the integrand is nonnegative, we have
\[
	e^{\Delta_{\bar x}v(z)} - e^{\Delta_{\bar y}v(z)} 
	\geq  D\psi_\beta(\bar x)\cdot z 
	\geq - C \beta |z|,
\]
while, by the mean value theorem, there exists $\xi\in\R$ with $|\xi|\leq \osc(u)$ such that
\[
e^{\Delta_{\bar x}v(z)} - e^{\Delta_{\bar y}v(z)} =
e^{\xi}  \bigl( \Delta_{\bar x}v(z) - \Delta_{\bar y}v(z)\bigr).
\]
Combining these inequalities with \eqref{eq:delta-estimates}, we obtain the following bounds, for all $z\in \mathcal{\tilde P}_+$:
\begin{equation}\label{eq:delta-estimates-abs}
	\bigl| \Delta_{\bar x}v(z) \bigr|, \bigl| \Delta_{\bar y}v(z) \bigr|
	\le  C\left(L |\bar x - \bar y|^{\theta-1} + \beta \right)|z|.
\end{equation}
Using~\eqref{eq:delta-max}, it follows from Taylor's theorem that, for all $z\in\mathcal{\tilde P_+}$
\begin{eqnarray*}
	e^{\Delta_{\bar x}v(z)} - e^{\Delta_{\bar y}v(z)} -  D\psi_\beta(\bar x) \cdot z
	& = &
	e^{\Delta_{\bar y}v(z)} \Bigl( e^{\Delta_{\bar x}v(z)-\Delta_{\bar y}v(z)}-1\Bigr)  - D\psi_\beta(\bar x) \cdot z \\
	& \leq & \left( e^{\Delta_{\bar y}v(z)} - 1\right) \Bigl( e^{\Delta_x\psi_\beta(z)}-1\Bigr) \\
	& & 
	+ \left(e^{\Delta_x\psi_\beta(z)}-1-\Delta_x\psi_\beta(z)\right)
	+ \bigl(\Delta_x\psi_\beta(z) - D\psi_\beta(\bar x) \cdot z \bigr)\\
	& \le & 
	C\left( e^{2\osc(u)+1} \bigl| \Delta_{\bar y}v(z) \bigr|\, \bigl| \Delta_x\psi_\beta (z) \bigr| 
	+ e^{|\Delta_x\psi_\beta (z)|} |\Delta_x\psi_\beta(z)|^2 + \beta^2 |z|^2 \right)\\
	& \le & 
	C  e^{2\osc(u)+1} \left( \beta L |\bar x - \bar y|^{\theta-1} |z|^2\, + \beta^2 |z|^2 \right).
\end{eqnarray*}
In view of assumption~(L1), it follows upon integration that
\begin{eqnarray*}
	\tilde\J 
	& \leq & 
	C e^{2\osc(u)+1} \left(\beta L |\bar x - \bar y|^{\theta-1} + \beta^2 \right)\int_{\mathcal{\tilde P}_+} |z|^2 \tilde \nu(dz)\\
	& \leq & 
	C\left( \beta L |\bar x - \bar y|^{\theta-1} + o_\beta(1) \right).
\end{eqnarray*}

\medskip
\noindent\emph{Estimates of $\J^\pm$.}
As the terms $\J^\pm$ are treated similarly, we focus on the estimate of $\J^+$ which contains the localization.
Proceeding as in the estimate of $\tilde\J$, we consider the measurable set $\mathcal P_+$ where the integrand
$\ell_{\bar x} v(z;p+q)$ in $\J^+$ is nonnegative, i.e., for all $z \in \mathcal P_+$
\[
	e^{\Delta_{\bar x}v(z)} - 1 \ge (p+q)\cdot z \ge -C\left(L |\bar x - \bar y|^{\theta-1} + \beta \right)|z|
\]
while, by the mean value theorem, there exists $c\in\R$ with $|c|\leq \osc(u)$ such that 
\[
	e^{\Delta_{\bar x}v(z)} - 1 = e^c\;\Delta_{\bar x}v(z).
\]
This implies that, for all $z\in\mathcal{P_+}$,
\[
	|\Delta_{\bar x}v(z)| \leq 
	C  e^{\osc(u)} \left(L |\bar x - \bar y|^{\theta-1} + \beta \right)|z|.
\]

Now, a further splitting of $B\setminus B_\delta$ is required. We consider
$\delta < \rho < |\bar x-\bar y|/2$.
The above, together with \eqref{eq:delta-max} and \eqref{eq:delta-est-order2}, leads to the following estimates. 
\begin{itemize}
\item For all $z\in B_{\rho}\setminus B_\delta$,
\begin{eqnarray*}
	\ell_{\bar x} v(z;p+q) 
	& = & \left(e^{\Delta_{\bar x}v(z)} - 1 -\Delta_{\bar x}v(z)\right)
	+ \Delta_{\bar x}v(z) - (p+q)\cdot z\\
	& \leq &   
	C e^{\osc(u)} |\Delta_{\bar x}v(z)|^2 
	+ \bigl(\Delta_{\bar x}\phi(z,\bar y) - p \cdot z \bigr)
	+ \bigl(\Delta_{\bar x}\psi_\beta(z) - q \cdot z \bigr)\\
	& \leq &   
	C e^{\osc(u)} \left( \left(L |\bar x - \bar y|^{\theta-1} + \beta \right)^2|z|^2 
	+ L|\bar x-\bar y|^{\theta-2}|z|^2 + \beta^2 |z|^2 \right)\\
	& \leq &   
	C e^{\osc(u)} \left( L^2 |\bar x - \bar y|^{2(\theta-1)} 
	+ L|\bar x-\bar y|^{\theta-2}+ o_\beta(1)\right)|z|^2.
\end{eqnarray*}

\item For all $z\in B\setminus B_{\rho}$, with $\rho < |\bar x-\bar y|/2$,
\begin{eqnarray*}
	\ell_{\bar x} v(z;p+q) 
	& = & \left(e^{\Delta_{\bar x}v(z)} - 1 \right) - (p+q)\cdot z 
	 \;\leq \;
	C e^{\osc(u)} |\Delta_{\bar x}v(z)|+ (|p|+|q|)\; |z|\\
	& \leq &   
	C e^{\osc(u)} \bigl( \left(L |\bar x - \bar y|^{\theta-1} + \beta \right)|z| 
	+  L|\bar x-\bar y|^{\theta-1}|z| + \beta |z| \bigr)\\
	& \leq &   
	C e^{\osc(u)} \left( L |\bar x - \bar y|^{(\theta-1)} + o_\beta(1)\right)|z|.
\end{eqnarray*}
\end{itemize}
Gathering the above estimates and recalling that  $\nu^+ \le |\nu_{\bar x} - \nu_{\bar y}|$, we get
\begin{eqnarray*}
	\J^+  &\le & 
	C e^{\osc(u)} \left( L^2 |\bar x - \bar y|^{2(\theta-1)}  
	+ L|\bar x-\bar y|^{\theta-2} + o_\beta(1)\right) 
	\int_{\mathcal {P_+}\cap (B_\rho\setminus B_\delta)}|z|^2  \;|\nu_{\bar x} - \nu_{\bar y}|(dz)  \\
	&& + 
	C e^{\osc(u)} \left( L |\bar x - \bar y|^{(\theta-1)} + o_\beta(1)\right)
	\int_{\mathcal {P_+}\cap (B\setminus B_\rho)}|z|\;|\nu_{\bar x} - \nu_{\bar y}|(dz)\\
	&\leq&
	C e^{\osc(u)  \left( L^2|\bar x-\bar y|^{2\theta -1} \rho^{2-\sigma}
	+  L|\bar x-\bar y|^{\theta -1} \rho^{2-\sigma} +  L|\bar x-\bar y|^{\theta}\zeta_\sigma(\rho)
	+  o_\beta(1)\zeta_\sigma(\rho)\right)},
\end{eqnarray*}
where we use assumption~(L3)(i)--(ii) to get the last inequality. Letting $\rho \to |{\bar x} - {\bar y}|/2$, we obtain
\begin{eqnarray*}
	\J^+ 
	&\le & 
	C  e^{\osc(u)} \left( L^2 |\bar x - \bar y|^{2\theta+1-\sigma}  
	+ L|\bar x-\bar y|^{\theta+1-\sigma} + L|\bar x - \bar y|^{\theta}\zeta_\sigma(|\bar x - \bar y|)\right)  
	+ o_\beta(1) \zeta_\sigma(|\bar x - \bar y|).
\end{eqnarray*}
As already mentioned, a similar estimate can be found for $\J^-$. Gathering all the estimates and recalling \eqref{eq:L-theta}, we reach the conclusion.
\end{proof}
\smallskip

\paragraph{\bf Conclusion.}
Gathering estimates~\eqref{eq:Hest}, \eqref{eq:NLest-in}, \eqref{eq:NLest-out}, and \eqref{NLest-crown}, and inserting them into the viscosity inequality \eqref{eq:visc-ineq}, we obtain
\begin{eqnarray*}
	L^{m+1} |\bar x-\bar y|^{\theta+m(\theta-1)} 
	&\leq & 
	o_\beta(1)|\bar x - \bar y|^{-(m-1)}
	+ C \Bigl(|\bar x-\bar y|^{-2}+o_\beta(1)\Bigr)\,o^{\bar x,\bar y}_\delta(1) 
	+ C|\bar x-\bar y| \\ 
	& & 
	+\, CL |\bar x-\bar y|^{\theta}\Bigl(\zeta_\sigma(|\bar x-\bar y|)
	+ \beta |\bar x-\bar y|^{-1} \Bigr)	+o_\beta(1)\zeta_\sigma(|\bar x - \bar y|).
\end{eqnarray*}
Since $|\bar x - \bar y|$ does not depend on $\delta$, we can first let $\delta \to 0$ to arrive at
\begin{eqnarray}\nonumber
	L^{m+1} 
	&\leq & 
	o_\beta(1) \left( |\bar x - \bar y|^{1-\theta (m+1)} + |\bar x-\bar y|^{m-\theta(m+1)}\zeta_\sigma(|\bar x-\bar y|)
	+ L |\bar x-\bar y|^{m-1-\theta m} \right)\\ 
	\label{estim-contr}
	&& 
	+  C|\bar x-\bar y|^{(m+1)(1-\theta)} + CL  |\bar x-\bar y|^{m-\theta m}\zeta_\sigma(|\bar x-\bar y|)
\end{eqnarray}

We proceed in two steps. We first deduce a non-optimal H\"older regularity and then we improve it. At first, we choose 
\[ \theta=\underline\theta < \min\left(\frac{1}{m+1}, \frac{m-1}{m}, \frac{m+1-\sigma}{m+1}\right)\]
in order that all the exponents of the terms $|\bar x-\bar y|$ in~\eqref{estim-contr} are positive, regardeless the value of $\sigma\in (0,2)$. 
By~\eqref{eq:L-theta}, we obtain a contradiction for $L$ large enough depending only on the data and $\osc(u)$, but not on $\beta\in (0,1)$. This proves that $v$ is $\underline\theta$-H\"older continuous with seminorm $\underline C$.

We then repeat the above proof taking into account that now, from~\eqref{eq:reg-sup} and the $\underline\theta$-H\"older continuity of $v$, we have
\begin{eqnarray*}
	3\varepsilon_L \leq v(\bar x)-v(\bar y)\leq \underline C|\bar x-\bar y|^{\underline\theta},
\end{eqnarray*}
hence
\begin{eqnarray*}
	|\bar x-\bar y|\geq \left(\frac{3\varepsilon_L }{\underline C}\right)^{1/ \underline\theta}=:\eta_L >0,
\end{eqnarray*}
where $\eta_L$ is independent of the localization parameter $\beta$. Moreover, $|\bar x-\bar y|\leq 1$ for $L$ large enough. Coming back to~\eqref{estim-contr} with this additional information, we let $\beta \to 0$. Denoting by $A_L \in [\eta_L,1]$ any limit of $|\bar x - \bar y|$, we obtain
\begin{equation}\label{eq:est-final}
	L^{m+1}\leq C\left(1+ L A_L^{m(1-\theta)} \zeta_\sigma(A_L)\right).
\end{equation}
We are now in a position to conclude, according to the range of $\sigma$.
\begin{itemize}
\item {\em Case~1: $0<\sigma<1$.}
	In this case $\zeta_\sigma(A_L) = A_L^{1-\sigma}$, and \eqref{eq:est-final} reads
	\begin{equation}\label{ineq532}
		L^{m+1}\leq C\left(1+ L A_L^{1-\sigma +m(1-\theta)}\right).
	\end{equation}
	Since $1-\sigma + m(1-\theta) > 0$ for all $\theta \leq 1$, we may choose $\theta=1$, 
	which yields $L^{m+1}\leq C(1+L)$, which is a contradiction for $L$ large enough. 
	It follows that $v$ is Lipschitz continuous.
	
\item {\em Case~2: $\sigma=1$.}
	In this case, $\zeta_\sigma(A_L) = |\log A_L|$, and \eqref{eq:est-final} reads
	\begin{equation*}
		L^{m+1}\leq C\left(1+ L A_L^{m(1-\theta)}|\log A_L|\right).
	\end{equation*}
	Since, for all $\theta <1$, there exists $C_\theta$ such that 
	$A_L^{m(1-\theta)}|\log A_L|\leq C_\theta$ for all $A_L\in (0,1]$, 
	we obtain again a contradiction for $L=L_\theta$ large enough, 
	and it follows that $v$ is $\theta$--H\"older continuous for any $\theta<1$.
	
\item {\em Case~3: $1<\sigma<2$.}
	In this case, $\zeta_\sigma(A_L) = A_L^{1-\sigma}$, and \eqref{eq:est-final} reads~\eqref{ineq532} as in case 1.
	The term  $A_L^{1-\sigma +m(1-\theta)}$ is bounded provided that \(1-\sigma + m(1-\theta)\geq 0.\)
	In particular, this holds for the largest admissible exponent
	\[
	\theta = \frac{m-\sigma + 1}{m}.
	\]
	This yields a contradiction, and we conclude that the solution is $\theta$--H\"older continuous, with $\theta$ as above.
\end{itemize}

This completes the proof of Theorem~\ref{thm:reg-Hoe}.
\end{proof}

\section{Proof of Lipschitz regularity--Theorem~\ref{thm:reg-Lip}}
\label{sec:pf-lip}

In view of Theorem~\ref{thm:reg-Hoe}, it remains to prove the Lipschitz continuity of the solution for $\sigma \geq 1$. We already know that the solution is H\"older continuous with exponent $\theta = (m-\sigma+1)/m$ when $\sigma>1$,  and with any exponent $\theta\in(0,1)$ when $\sigma=1$. The strategy is to employ the ellipticity assumption~(L4) and upgrade the H\"older to Lipschitz continuity.
\smallskip

We rely on the Ishii--Lions method \cite{IL90,BCCI12} and argue by contradiction. Assume that
\begin{eqnarray}\label{eq:sup-Lip}
	M = \sup_{\R^N\times\R^N} \{u(x) - u(y) - \phi(x, y)\} > 0,
\end{eqnarray}
where $\phi(x,y) = \varphi(|x-y|)$. Here $\phi$ {is} redefined for a new, more general $\varphi:\R_+ \to \R_+$ that is a concave function such that $\varphi$ is smooth and increasing on $[0,t_0)$, with $\varphi(0)=0$, and $\varphi(t)=\varphi(t_0) \geq \osc(u)+1$ for all $t\geq t_0$. The constant  $t_0\in(0,1)$ and the function $\varphi$ will be specified below in the statement of Lemmas~\ref{lem:Lip-NL-est} and~\ref{lem:Lip-NL-est1}.

To avoid technicalities, we omit the localization (follow the previous proof in Section~\ref{sec:Hoe-reg} for details on this issue) and assume that the above supremum is attained at some $\bar x,\bar y \in \R^N$. Note that, thanks to~\eqref{eq:sup-Lip} and the assumptions on $\varphi$, we have $0<|\bar x-\bar y|\leq t_0$. 

Denote
\[
p := D_x\phi(\bar x,\bar y) = -D_y\phi(\bar x,\bar y)
= \varphi'(|\bar x-\bar y|)
  \frac{\bar x-\bar y}{|\bar x-\bar y|}.
\]

Writing the viscosity inequalities, respectively for the subsolution $u$ at
$\bar x$ and the supersolution $u$ at $\bar y$, and subtracting them, we obtain
that, for any $0<\delta<1$,
\begin{equation}\label{Lip-visc-ineq}
\begin{split}
    \lambda \bigl(u(\bar x)-u(\bar y)\bigr)
    + H(\bar x,p) - H(\bar y,p)
    \leq
        &\quad \I[B_\delta](\bar x,\phi(\cdot,\bar y))
              - \I[B_\delta](\bar y,-\phi(\bar x,\cdot))\\
        &\quad  + \I[B\setminus B_\delta](\bar x,u,p)
              - \I[B\setminus B_\delta](\bar y,u,p) \\
        &\quad + \I[B^c](\bar x,u)
              - \I[B^c](\bar y,u).
\end{split}
\end{equation}
In the following, we estimate the terms appearing in~\eqref{Lip-visc-ineq} in
order to reach a contradiction.

In view of assumption (H1) and the concavity of $\varphi$, we have
\begin{eqnarray*}
    H(\bar y,p) - H(\bar x,p)
    \leq
        C_H |\bar x-\bar y|
        \bigl(1+\left(\varphi'(|\bar x-\bar y|)\right)^m\bigr)
    \leq
        C_H |\bar x-\bar y|
        \left(
            1+
            \left(
                \frac{\varphi(|\bar x-\bar y|)}{|\bar x-\bar y|}
            \right)^m
        \right).
\end{eqnarray*}
From Theorem~\ref{thm:reg-Hoe}, the solution is $\theta$-H\"older continuous and
there exists $L_\theta$ (depending on $\osc(u)$) such that
\begin{equation}\label{phi-theta}
    \varphi(|\bar x-\bar y|) 
    \leq u(\bar x) - u(\bar y)
    \leq L_\theta |\bar x-\bar y|^\theta.
\end{equation}
This implies that
\begin{equation}\label{Lip-H-est}
    H(\bar y,p) - H(\bar x,p)
    \leq
    C |\bar x-\bar y|
    \left(1+ L_\theta^m|\bar x-\bar y|^{m(\theta-1)}\right).
\end{equation}

By arguments similar to those in Lemma~\ref{lem:NLest-in} and
Lemma~\ref{lem:NLest-out}, and in view of assumptions~(L1),~(L2),
and~(L3)(iii), there exists a constant $C>0$, depending on $\osc(u)$, such
that
\begin{equation}\label{Lip-NLest-in}
    \I[B_\delta](\bar x,\phi(\cdot,\bar y))
    - \I[B_\delta](\bar y,-\phi(\bar x,\cdot))
    \le o_{\delta}^{\bar x,\bar y}(1),
\end{equation}
and
\begin{equation}\label{Lip-NLest-out}
    \I[B^c](\bar x,u) - \I[B^c](\bar y,u)
    \le C|\bar x-\bar y|.
\end{equation}
Letting first $\delta \searrow 0$, the first term above vanishes. 

We then estimate the nonlocal difference on the set $B^*$. 
To this end, we use the weak ellipticity assumption~(L4) and introduce the ellipticity cone $\mathcal C^{\eta}_{\rho}(p)$. 
We split the integration domain into the three regions
\(
\mathcal C^{\eta}_{\rho}(p),\;
B_{\rho}\setminus \mathcal C^{\eta}_{\rho}(p),\;
B\setminus B_{\rho},
\)
and, for any measurable set $A\subset \mathbb R^N$, we define
\[
\mathcal T[A] := \mathcal I[A](\bar x,u,p)-\mathcal I[A](\bar y,u,p).
\]
Accordingly, the nonlocal difference over $B^*$ can be decomposed as
\begin{eqnarray*}
\mathcal T[B^*]
&=&
	\mathcal T[\mathcal C^{\eta}_{\rho}(p)] + 
	\mathcal T[B_{\rho}\setminus \mathcal C^{\eta}_{\rho}(p)] + 
	\mathcal T[B\setminus B_{\rho}].
\end{eqnarray*}

The key point is that, thanks to assumption~(L4), the contribution over the ellipticity cone is negative and dominates the other two terms. More precisely, it was shown in \cite{BCCI12,CGT22} that there exist an exponent $\gamma>0$ (depending on $\sigma\in[1,2)$ and on the choice of $\varphi$) and a constant {$K_0>0$} such that
\begin{eqnarray*}
	\mathcal T[\mathcal C^{\eta}_{\rho}(p)]
	&\leq&
	-K_0 L|\bar x-\bar y|^\gamma,
\end{eqnarray*}
whereas the remaining contributions are of lower order:
\begin{eqnarray*}
	\mathcal T[B_{\rho}\setminus \mathcal C^{\eta}_{\rho}(p)] + 
	\mathcal T[B\setminus B_{\rho}]
&\leq&
	L|\bar x-\bar y|^\gamma\,o_{|\bar x-\bar y|}(1).
\end{eqnarray*}
The argument relies on the strict concavity of $\varphi$ on $(0, t_0)$ for some $t_0 > 0$, which in turn implies the concavity of $\phi$ near the origin in the direction of the gradient $p$. 
Choosing a small radius $\rho$ and a small cone aperture $\eta$, one obtains a negative integrand on $\mathcal C^{\eta}_{\rho}(p)$, uniformly bounded away from zero. The nondegeneracy of the measure then yields the desired estimate.

For the reader's convenience, we recall the precise bounds established in \cite[Lemmas~12 and~13]{BCCI12} in the case $\sigma>1$, and in \cite[Lemmas~3.2--3.4]{CGT22} in the case $\sigma=1$.

\begin{lemma}[Case $\sigma>1$]\label{lem:Lip-NL-est}
Assume that~(L1)--(L4) hold with $\sigma>1$. Let
\begin{equation*}
\varphi(t) =
    \left\{
    \begin{array}{ll}
        L\bigl(t-c_0t^{1+\theta_0}\bigr), & t\in[0,t_0],\\
        \varphi(t_0), & t>t_0,
    \end{array}
    \right.
\end{equation*}
where $\theta_0\in(0,1]$ is small enough, $c_0 > 1/\theta_0 2^{\theta_0 - 1}$, $t_0 = \mathrm{arg}\max_{t \in (0, 1)} \big\{ t - c_0 t^{1 + \theta_0} \big\}$ and  $L > 0$ is large enough depending on given data and $\osc(u)$.
Then there exists a constant $K_0>0$, depending only on the data and $\osc(u)$, such that
\begin{eqnarray}\label{Lip-NL>1}
    \mathcal T[B^*]
    \leq -L|\bar x-\bar y|^{\gamma}
        \bigl(K_0+o_{|\bar x-\bar y|}(1)\bigr),
\end{eqnarray}
where $\gamma=(1-\sigma)+\theta_0(N+2-\sigma)$.
\end{lemma}



\begin{lemma}[Case $\sigma=1$]\label{lem:Lip-NL-est1}
Assume that~(L1)--(L4) hold with $\sigma=1$. Let
\begin{equation*}
\varphi(t) =
    \left\{
    \begin{array}{ll}
        0, & t=0,\\
        L\bigl(t+t\log^{-1}(t)\bigr), & t\in(0,t_0],\\
        \varphi(t_0), & t\ge t_0,
    \end{array}
    \right.
\end{equation*}
where $t_0=\arg\max_{t\in (0,1)}\{ t+t\log^{-1}(t)\}$ and $L\geq (\osc(u)+1) (t_0+t_0 \log^{-1}(t_0))^{-1}$. Then there exists a constant $K_0>0$, depending only on the data, such that
\begin{eqnarray}\label{Lip-NL-est1}
    \mathcal T[B^*]
    \le -L\; \left| \log |\bar x-\bar y| \right|^{-(N+3)} \bigl(K_0 + o_{|\bar x-\bar y|}(1)\bigr).
\end{eqnarray}
\end{lemma}

We are now in a position to conclude the proof of Theorem~\ref{thm:reg-Lip}. 
Recall that $M>0$ in~\eqref{eq:sup-Lip} and that $(\bar x,\bar y)$ is a maximum point satisfying  $0<|\bar x-\bar y|\to 0$ as $L\to\infty$.

\begin{itemize}
\item {\em Case 1: $\sigma>1$.}
	By Theorem~\ref{thm:reg-Hoe}, $u$ is H\"older continuous with exponent
	\(
	\displaystyle \theta=\frac{m-\sigma+1}{m}.
	\)
	Substituting estimates~\eqref{Lip-H-est},~\eqref{Lip-NLest-out} and~\eqref{Lip-NL>1}
	into the viscosity inequality~\eqref{Lip-visc-ineq}, we obtain
	\begin{eqnarray*}
	\lambda M
	&\le&
	- L|\bar x-\bar y|^{\gamma} \bigl(K_0-o_{|\bar x-\bar y|}(1)\bigr)
	+ C |\bar x-\bar y|^{m(\theta-1)+1}+ o_{|\bar x-\bar y|}(1) \\
	&\le&
	-L|\bar x-\bar y|^{\gamma}
	\Bigl(K_0 - C_H |\bar x-\bar y|^{2-\sigma-\gamma} 
	- o_{|\bar x-\bar y|}(1) \Bigr) + o_{|\bar x-\bar y|}(1).
\end{eqnarray*}
Choosing $\theta_0>0$ sufficiently small so that $ 2-\sigma-\gamma = 1-\theta_0(N+2-\sigma) >0,$ it follows that the previous inequality reduces to
\begin{eqnarray*}
	\lambda M
	&\le& -L|\bar x-\bar y|^{\gamma} \bigl(K_0-o_{|\bar x-\bar y|}(1)\bigr) 
	+ o_{|\bar x-\bar y|}(1).
\end{eqnarray*}
Since {$K_0>0$ and $o_{|\bar x-\bar y|}(1)\to 0$ as $L\to\infty$, the right-hand side is strictly negative for $L$ large enough, which} contradicts the fact that \(\lambda M\ge 0\).
\smallskip

\item {\em Case 2: $\sigma=1$.}
	Substituting estimates~\eqref{Lip-H-est},~\eqref{Lip-NLest-out} and~\eqref{Lip-NL-est1} 
	into the viscosity inequality~\eqref{Lip-visc-ineq}, and recalling that, by
	Theorem~\ref{thm:reg-Hoe}, the function $u$ is $\theta$--H\"older
	continuous for any exponent $\theta\in(0,1)$, we obtain
	\begin{eqnarray*}
	\lambda M
	&\le&
	- L\bigl|\log |\bar x-\bar y|\bigr|^{-(N+3)} \Bigl[ K_0-o_{|\bar x-\bar y|}(1)
	+ C|\bar x-\bar y| \bigl|\log |\bar x-\bar y|\bigr|^{N+3}\\
	&&
	+ C|\bar x-\bar y|^{m(\theta -1) +1}\bigl|\log |\bar x-\bar y|\bigr|^{N+3} \Bigr]
	+ o_{|\bar x-\bar y|}(1).
\end{eqnarray*}
Choose $\theta>(m-1)/m$, so that $m(\theta -1) +1>0$. Then the inequality reduces to
\begin{eqnarray*}
	\lambda M
	&\le&
	- L\bigl|\log |\bar x-\bar y| \bigr|^{-(N+3)} \bigl(K_0-o_{|\bar x-\bar y|}(1)\bigr)
	+ o_{|\bar x-\bar y|}(1).
\end{eqnarray*}
Since {$K_0>0$} and $o_{|\bar x-\bar y|}(1)\to 0$ as $L\to +\infty$, the right-hand side is strictly negative for $L$ large enough, which again contradicts \(\lambda M\ge 0\).
\end{itemize}

This completes the proof of Theorem~\ref{thm:reg-Lip}.
\hfill$\Box$


\section{Non--Lipschitz solutions under weaker assumptions}
\label{appendix}

This final section shows that the hypotheses of Theorems \ref{thm:reg-Hoe} and  \ref{thm:reg-Lip} cannot be relaxed further without losing Lipschitz regularity. We proceed in three steps. First, we construct an \emph{unbounded} Hölder continuous viscosity solution for a family of measures satisfying the weaker assumption~(L3') but violating both~(L1) and~(L2). We then truncate this solution to obtain a \emph{bounded} Hölder continuous solution for a family of measures satisfying~(L1),~(L2), and the weaker assumption~(L3'), which requires only Hölder, rather than Lipschitz, dependence of the measure on the space variable $x$. Finally, we adjust the far-field value of the truncated solution so that the data becomes Lipschitz, while the solution stays non-Lipschitz. \medskip

We weaken~(L3) to the following Hölder continuity assumption:
\begin{itemize}
\item[(L3')] There exist a constant $C_\nu>0$, an exponent $\gamma\in(0,1]$, and a nonlocal order $\sigma\in(0,2)$ such that, for all $\delta\in(0,1)$ and all $x,y\in\R^N$,
	\[ \begin{aligned}
		(i)\quad &
		\int_{B_\delta} |z|^2\,|\nu_x-\nu_y|(dz) \le C_\nu |x-y|^\gamma\,\delta^{2-\sigma}, \\
		(ii)\quad &
	\int_{B\setminus B_\delta} |z|\,|\nu_x-\nu_y|(dz) \le C_\nu |x-y|^\gamma
		\zeta_\sigma(\delta)
		\qquad \text{ with }
		\zeta_\sigma(\delta)=
		\begin{cases}
			\delta^{1-\sigma}, & \text{if } \sigma\neq1,\\
			|\log\delta|, & \text{if } \sigma=1,
		\end{cases}  \\
	(iii)\quad &
	\int_{B_\delta^c} |\nu_x-\nu_y|(dz) \le C_\nu |x-y|^\gamma\,{\delta^{-\sigma}}.
	\end{aligned}
	\]
\end{itemize}

We recall that, for $\sigma\in(0,2)$ and $u:\R\to\R$ smooth enough and with suitable growth at infinity, the fractional Laplacian 
$(-\Delta)^{\sigma/2}u$ can be defined equivalently as
\[
(-\Delta)^{\sigma/2} u(x) = C_{1,\sigma}\,\mathrm{P.V.}\!\int_{\R}\frac{u(x)-u(y)}{|x-y|^{1+\sigma}}\,dy
= -\,C_{1,\sigma}\,\mathrm{P.V.}\!\int_{\R}\bigl(u(x+z)-u(x)\bigr)\,\frac{dz}{|z|^{1+\sigma}},
\]
where $C_{1,\sigma}$ is a positive renormalisation constant, expressed in terms of Gamma functions.
\medskip

Throughout this section, $x_+^\beta$ denotes the function equal to $x^\beta$ for $x>0$ and to $0$
for $x\le0$, for any $\beta\in\mathbb R$.

\begin{prop}\label{prop:unbd-sol}
Assume $m>1$ and $\sigma\in(1,2)$, and let $\gamma_0={(m(\sigma-2)+\sigma)}/{2}$.
For every $\gamma\in( (\gamma_0)_+,\sigma-1)$ there exist $\theta\in(\sigma/2,1)$ and $A>0$ such that
\begin{equation}\label{eq:cex-sol}
	u(x) = A\,x_+^\theta,
\end{equation}
is an unbounded viscosity solution of
\begin{equation}\label{eq:cexHJ}
	x_+^\gamma(-\Delta)^{\sigma/2}u + |u'|^m = 0 \quad\text{in}\quad \mathbb R.
\end{equation}
\end{prop}

\begin{proof}
Let $u$ be given by \eqref{eq:cex-sol}, with $\theta$ and $A$ yet to be determined.

\emph{Step 1: particular solutions for the fractional Laplacian.}
It can be checked by direct computation (see~\cite{Dyda12},~\cite[Lemma 2.1]{adv25}), that  for every $\theta\in(-1,\sigma)$, we have, for $x>0$,
\begin{equation}\label{eq:frac-power}
(-\Delta)^{\sigma/2}(x_+^\theta)(x) = -\,c_\sigma(\theta)\,x^{\theta-\sigma},
\end{equation}
where the constant satisfies 
$$ c_\sigma(\theta)<0 \quad \text{ for $\theta\in\left(0,\sigma/2\right)$}, \qquad
c_\sigma(\theta)>0  \quad \text{ for $\theta\in\left(\sigma/2,\sigma\right)$.}$$
At $\theta=\sigma/2$, $c_\sigma(\sigma/2)=0$ and we recover the classical fact that 
$x_+^{\sigma/2}$ is $\sigma/2$-harmonic.

\smallskip
\emph{Step 2: $u$ is a viscosity solution of equation} \eqref{eq:cexHJ}.
For $x>0$, we have, in view of Step 1 above,
\[
	x_+^\gamma(-\Delta)^{\sigma/2}u(x) + |u'(x)|^m = 
	- A\,c_\sigma(\theta)\,x^{\gamma+\theta-\sigma}  + (A\theta)^m x^{m(\theta-1)}.
\]
Choose $\theta$ so that the two exponents coincide: $\gamma+\theta-\sigma=m(\theta-1)$.  This gives
\begin{equation}\label{eq:cex-theta}
	 \theta = \frac{m-\sigma+\gamma}{m-1}.
\end{equation}
Note that 
\(\theta \in (\sigma/2,1) \text{ if and only if } \gamma\in \bigl(\gamma_0, \sigma-1\bigr).\)
Therefore, $u$ satisfies
\[
	x_+^\gamma(-\Delta)^{\sigma/2}u(x)+|u'(x)|^m =
	x^{\gamma+\theta-\sigma}\Big(A^m\theta^m - A\,c_\sigma(\theta)\Big),
\]
with the constant $c_\sigma(\theta)>0$.  We then choose $A$ to be the positive root of the bracket on the right-hand side:
\begin{equation}\label{eq:cex-A}
	A=\left(\frac{c_\sigma(\theta)}{\theta^m}\right)^{1/(m-1)}>0.
\end{equation}
Thus, $u$ solves \eqref{eq:cexHJ} classically for every $x>0$.

For $x<0$, by the convention above, $x_+^\gamma\equiv0$ on $\{x<0\}$, whereas $u'(x)=0$. Hence, both terms of \eqref{eq:cexHJ} vanish identically, and the equation holds trivially for $x<0$.

At $x=0$, we check that the equation is satisfied in the viscosity sense only. Since $\theta<1$, no test function touches $u$ from above at $0$, and $u$ is trivially a viscosity subsolution there. For the supersolution property, let $\phi\in C^2$ satisfy $\phi(0)=0$ and $\phi\le u$ near $0$. Since $x_+^\gamma\big|_{x=0}=0$,
\[
	x_+^\gamma(-\Delta)^{\sigma/2}\phi(x)\Big|_{x=0} + |\phi'(0)|^m = |\phi'(0)|^m\ge0,
\]
which is exactly the required supersolution inequality at $0$.
\end{proof}

\begin{remark}
No analogue of this construction exists in the local case $\sigma=2$. Indeed, the zeros of
$c_\sigma(\theta)$ occur at $\theta=\sigma/2$ and $\theta=\sigma/2-1$;
at $\sigma=2$ these collapse exactly onto the endpoints $\theta=1$ and $\theta=0$ of the interval
$(0,1)$, so no sign change of $c_\sigma$ occurs \emph{strictly inside} $(0,1)$. Concretely, for
$u=x_+^\theta$ and $\sigma=2$,
\[
-\Delta u = -u'' = \theta(1-\theta)\,x_+^{\theta-2} > 0 \qquad\text{for every }\theta\in(0,1),
\]
consistently with $c_\sigma(\theta)=\theta(\theta-1)<0$ throughout $(0,1)$. 
Consequently $-\Delta u + |u'|^m$ has a fixed sign for all $\theta\in(0,1)$, and no choice of
$A>0$ can balance the two terms. The non-Lipschitz counterexample is therefore a genuinely
nonlocal phenomenon.
\end{remark}

\begin{remark}
Let $(\nu_x)_{x\in\R}$ be the family of measures naturally associated with the operator in Proposition~\ref{prop:unbd-sol} via
\[
	x_+^\gamma(-\Delta)^{\sigma/2}u(x) = \mathrm{P.V.}\!\int_\R \bigl(u(x)-u(x+z)\bigr)\,\nu_x(dz),
	\qquad
	\nu_x(dz) := C_{1,\sigma}\,x_+^\gamma\,|z|^{-1-\sigma}\,dz .
\]
Since $t\mapsto t_+^\gamma$ is globally $\gamma$-Hölder continuous on $\R$ for $\gamma\in(0,1]$, 
the family $(\nu_x)_{x\in\R}$ satisfies~(L3'), with the same nonlocal order $\sigma$ as in the kernel $|z|^{-1-\sigma}$ 
and Hölder exponent $\gamma$ as in Proposition~\ref{prop:unbd-sol}. 
However, $(\nu_x)_{x\in\R}$ satisfies \emph{neither}~(L1) \emph{nor}~(L2): both
\[
	\int_{B_1}|z|^2\,\nu_x(dz) = C_{1,\sigma}\,\frac{2}{2-\sigma}\,x_+^\gamma
	\qquad\text{and}\qquad
	\int_{\R\setminus B_R}\nu_x(dz) = \frac{C_{1,\sigma}}{\sigma}\,x_+^\gamma\,R^{-\sigma}
\]
grow without bound as $x\to+\infty$, so that the suprema over $x\in\R$ required in~(L1) and~(L2) are infinite. 
This is not merely a technical shortcoming of the example: it is precisely the mechanism responsible for 
the unboundedness of the explicit solution $u(x)=Ax_+^\theta$ constructed. 
This shows that~(L2), together with~(L1), cannot be dropped from the hypotheses without losing boundedness of viscosity solutions, even though the weaker Hölder continuity assumption~(L3') continues to hold.
\end{remark}

We next truncate the kernels to construct a \emph{bounded} solution, for a Hamilton--Jacobi equation with H\"older continuous data.

\begin{prop}\label{prop:bdd-sol}
Assume $m>1$ and $\sigma\in(1,2)$, and let $\gamma_0={(m(\sigma-2)+\sigma)}/{2}$.
Let $\chi\in C^\infty(\R)$, $0\le\chi\le1$, $\chi\equiv1$ on $(-\infty,2]$ and $\chi\equiv0$ on $[3,+\infty)$.
For every $\gamma\in((\gamma_0)_+,\sigma-1)$ there exist $\theta\in(\sigma/2,1)$ and $A>0$ such that
\begin{equation}\label{eq:cex-sol-h}
	u(x) = A\,x_+^\theta\chi(x),
\end{equation}
is a viscosity solution of
\begin{equation}\label{eq:cexHJ-h}
	\chi(x)\,x_+^\gamma(-\Delta)^{\sigma/2}u + |u'|^m = f(x) \quad\text{in}\quad \mathbb R,
\end{equation}
with $f\in C^{0,\gamma}(\R)$ compactly supported in $[0,3]$.
\end{prop}

\begin{proof}

\emph{Step 1: structure of $u$.}
We have $u\in C^\infty(\R\setminus\{0\})$, with
$$
u\equiv0 \ \text{on } (-\infty,0]\cup[3,+\infty), \qquad u(x)=Ax^\theta \ \text{on } [0,2].
$$
In particular $u$ has compact support $[0,3]$, and $u$ is exactly $\theta$-H\"older (not Lipschitz) at $x=0$.
Fix $\theta\in(\sigma/2,1)$  and $A>0$ as in Proposition \ref{prop:unbd-sol}.

We split $u$ into a clean power and its correction 
\[
u = A\,x_+^\theta + w, \qquad w(x):=A\,x_+^\theta\big(\chi(x)-1\big),
\]
with $w$ supported in $[2,+\infty)$ and $w\le0$. By linearity, for $x>0$,
\[
(-\Delta)^{\sigma/2}u(x) = (-\Delta)^{\sigma/2}(A x_+^\theta)(x) + (-\Delta)^{\sigma/2}w(x).
\]

\smallskip
\emph{Step 2: the equation on $\R\setminus\{0\}$, with $f$ to be specified.}

For $x\in(-\infty,0)$, $\chi(x)x_+^\gamma=0$ and $u'(x) = 0$, so the equation is trivially satisfied with $f(x)=0$.

Let $x\in(0,2)$. Since  $w\equiv 0$ on $(-\infty,2]$ and $w\leq 0$ on $[2,+\infty)$, we have
\[
	(-\Delta)^{\sigma/2}w(x) =
	C_{1,\sigma}\int_2^{+\infty}\frac{-w(y)}{(y-x)^{1+\sigma}}\,dy =: R(x), \qquad R(x)\ge0.
\]
Since $w(y)=-Ay^\theta$ for $y\ge3$ and $\theta<\sigma$, $R$ is well defined. Moreover, all derivatives of $\chi$ vanish at $2$. Hence
the singularity of the kernel at $y=x=2$ is compensated to all orders, giving
\[
	R\in C^\infty([0,2]),\qquad R(0)=A\,R_0, \quad
	R_0:=C_{1,\sigma}\int_2^{+\infty}y^\theta\big(1-\chi(y)\big)y^{-1-\sigma}\,dy\in(0,+\infty).
\]
By linearity and \eqref{eq:frac-power}, for $x\in(0,2)$,
\[
	(-\Delta)^{\sigma/2}u(x) = -A\,c_\sigma(\theta)\,x^{\theta-\sigma} + R(x).
\]
Since $u(x)=Ax^\theta$ on $[0,2]$, $u'(x)=A\theta x^{\theta-1}$, and using \eqref{eq:cex-theta} and \eqref{eq:cex-A},
\[
	\chi(x)\,x^\gamma(-\Delta)^{\sigma/2}u(x)+|u'(x)|^m
	= \big(A^m\theta^m - Ac_\sigma(\theta)\big)\,x^{\gamma+\theta-\sigma} + x^\gamma R(x)
	= x^\gamma R(x).
\]
Set 
\begin{equation}\label{eq:cex-f-h}
	f(x) := x^\gamma R(x), \qquad x\in(0,2).
\end{equation}

For $x\in[2,3]$, we have $u=A x^\theta\chi(x)\in C^\infty([2,3])$, being bounded away from the origin. Moreover, $(-\Delta)^{\sigma/2}u$ is smooth on $[2,3]$ by standard local regularity, with bounds depending only on $\theta,\sigma,A,\chi$. Set
\[
	f(x):=\chi(x)\,x^\gamma(-\Delta)^{\sigma/2}u(x)+|u'(x)|^m, \qquad x\in[2,3].
\]
and the limitation of the regularity of $f$ in $[2,3]$ comes from the term $|u'(x)|^m$.
Since $u'(2)>0$ and $u(2) > 0 = u(3)$, $u$ has an interior maximum at some $x_0\in(2,3)$ with $u'(x_0)=0$.  As $m>1$, $g(t):=|t|^m$ satisfies $g\in C^1(\R)$, with $g'(t)=m\,\mathrm{sgn}(t)|t|^{m-1}$ Hölder continuous of exponent $\min(1,m-1)$. Hence $|u'(\cdot)|^m\in C^{1,\min(1,m-1)}$ near $x_0$,  and so is $f$, the other term being smooth there.   At $x=3$, all derivatives of $u$, hence of $u'$, and of $\chi x^\gamma(-\Delta)^{\sigma/2}u$, vanish. Thus $f$ and all its derivatives vanish as $x\to3^-$, matching $f\equiv0$ beyond. In particular  $f\in C^{0,1}([2,3]).$

For $x\in[3,+\infty)$, we have $\chi(x)\equiv0$, so the nonlocal term vanishes and  $u\equiv0,\ u'\equiv0$ locally. Thus
\[
	f(x) := \chi(x)\,x^\gamma(-\Delta)^{\sigma/2}u(x)+|u'(x)|^m \equiv 0, \qquad x\ge3.
\]

Note that $f$ is compactly supported in $[0,3]$; in particular $f\in L^\infty(\R)$.

\smallskip
\emph{Step 3: viscosity solution property at $x=0$.}
Since $\theta<1$, no smooth test function touches $u$ from above at $0$, so $u$ is trivially a viscosity subsolution there. For the supersolution property, let $\phi\in C^2$ satisfy $\phi(0)=0$ and $\phi\le u$ near $0$. Since $x_+^\gamma\big|_{x=0}=0$,
\[
\chi(x) x_+^\gamma(-\Delta)^{\sigma/2}\phi(x)\Big|_{x=0} + |\phi'(0)|^m = |\phi'(0)|^m\ge 0 = f(0),
\]
which is exactly the required supersolution inequality at $0$.
	
\smallskip
\emph{Step 4: H\"older regularity of $f$.}
By the previous steps, $f\in C^{0,1}_{\mathrm{loc}}(\R\setminus\{0\})$. Near $x=0$, by~\eqref{eq:cex-f-h} and $R(0)=AR_0>0$,
\[
f(x)=AR_0\,x^\gamma+O(x^{\gamma+1}), \qquad x\to0^+,
\]
thus $f$ is continuous at $0$, matching $f(0)=0$, with leading behavior $x^\gamma$.  Therefore,
 $f$ is not Lipschitz continuous at $x=0$ but 
$f\in C^{0,\gamma}(\R)$.
\end{proof}

\begin{remark}
Let $(\nu_x)_{x\in\R}$ now be the family of measures associated with the truncated operator
of Proposition~\ref{prop:bdd-sol},
\[
	\chi(x)\,x_+^\gamma(-\Delta)^{\sigma/2}u(x)
	= \mathrm{P.V.}\!\int_\R \bigl(u(x)-u(x+z)\bigr)\,\nu_x(dz),
	\qquad
	\nu_x(dz) := C_{1,\sigma}\,\chi(x)\,x_+^\gamma\,|z|^{-1-\sigma}\,dz .
\]
In contrast to the unbounded example above, the coefficient $x\mapsto\chi(x)x_+^\gamma$ is now
bounded and compactly supported in $[0,3]$, with $\chi(x)x_+^\gamma\le3^\gamma$. Hence
\[
	\sup_{x\in\R}\int_{B_1}|z|^2\,\nu_x(dz)
	= C_{1,\sigma}\,\frac{2}{2-\sigma}\,\sup_{x\in\R}\chi(x)x_+^\gamma<\infty,
	\qquad
	\sup_{x\in\R}\int_{\R\setminus B_R}\nu_x(dz)
	= \frac{C_{1,\sigma}}{\sigma}\,R^{-\sigma}\sup_{x\in\R}\chi(x)x_+^\gamma<\infty,
\]
so that $(\nu_x)_{x\in\R}$ satisfies both~(L1) and~(L2). Moreover, since
$x\mapsto\chi(x)x_+^\gamma$ is globally $\gamma$-Hölder continuous on $\R$, the family also
satisfies~(L3'), with order $\sigma$ and Hölder exponent $\gamma$. It does \emph{not},
however, satisfy the Lipschitz assumption~(L3): near $x=0$ the coefficient behaves like
$x_+^\gamma$ with $\gamma<1$.  This points that Lipschitz dependence of the measure on $x$ cannot be dropped in favour of~(L3'), without losing Lipschitz regularity of viscosity solutions.
\end{remark}

We next show that the loss of Lipschitz regularity of the solution persists even against Lipschitz data $f$.
\begin{prop}\label{prop:lip-data}
Assume $m>1$ and $\sigma\in(1,2)$, and let $\gamma_0={(m(\sigma-2)+\sigma)}/{2}$.
Let $\chi$ be as in Proposition~\ref{prop:bdd-sol}. 
For every $\gamma\in( (\gamma_0)_+,\sigma-1)$, there exist $\theta\in(\sigma/2,1)$, $A>0$ and $M>0$, 
so that
\begin{equation}\label{eq:cex-sol-lip}
u(x) := \begin{cases}
A x_+^\theta, & x\le2,\\[2pt]
\chi(x)\,A\,x^\theta + \big(1-\chi(x)\big)\,M, & 2\le x\le3,\\[2pt]
M, & x\ge3,
\end{cases}
\end{equation} 
is a viscosity solution of
\begin{equation}\label{eq:cexHJ-lip}
	\chi(x)\,x_+^\gamma(-\Delta)^{\sigma/2}u + |u'|^m = f(x) \quad\text{in}\quad \mathbb R,
\end{equation}
with $f\in C^{0,1}(\R)$ compactly supported in $[0,3]$. 
\end{prop}

\begin{proof}
 
\emph{Step 1: structure of $u$.}
Fix $\theta\in(\sigma/2,1)$ and $A>0$ as in Proposition~\ref{prop:unbd-sol}; the constant
$M>0$ will be chosen in Step~2. Formula \eqref{eq:cex-sol-lip} is equivalent to the
single expression
\[
u(x)=\chi(x)\,A x_+^\theta+\big(1-\chi(x)\big)M, \qquad x\in\R.
\]
Because $\chi\in C^\infty(\R)$ has all derivatives vanishing at $x=2$ and $x=3$, this forces $u\in C^\infty(\R\setminus\{0\})$, with $u(x)=Ax^\theta$ on $[0,2]$ and $u\equiv M$ on $[3,+\infty)$. We split $u$ into a clean power and its correction
\[
u = A\,x_+^\theta + w, \qquad w(x):=\big(1-\chi(x)\big)\big(M-A x_+^\theta\big),
\]
with $w$ supported in $[2,+\infty)$ and $w\equiv0$ on $(-\infty,2]$. By linearity, for $x>0$,
\[
(-\Delta)^{\sigma/2}u(x) = (-\Delta)^{\sigma/2}(A x_+^\theta)(x) + (-\Delta)^{\sigma/2}w(x).
\]
 
\smallskip
\emph{Step 2: the equation on $\R\setminus\{0\}$, and the choice of $M$.}
For $x\in(-\infty,0)$, $\chi(x)x_+^\gamma=0$ and $u'(x)=0$, so the equation is trivially
satisfied with $f(x)=0$.
 
For $x\in(0,2)$, we have $w(x)=0$, and since $w(y)\sim -Ay^\theta$ as $y\to+\infty$
(the constant $M$ being a lower-order correction), it follows that
\[
(-\Delta)^{\sigma/2}w(x)
= C_{1,\sigma}\int_2^{+\infty}\frac{-w(y)}{(y-x)^{1+\sigma}}\,dy 
= C_{1,\sigma}\int_2^{+\infty}\frac{\big(1-\chi(y)\big)\big(Ay^\theta-M\big)}{(y-x)^{1+\sigma}}\,dy
=: R(x)
\]
is well defined since $\theta<\sigma$. As all derivatives of $\chi$ vanish at $2$, the
singularity of the kernel at $y=x=2$ is compensated to all orders, giving $R\in
C^\infty([0,2])$ with
\[
R(0)=A R_0-M R_1, \qquad
R_0 := C_{1,\sigma}\!\int_2^{+\infty}\! y^\theta\big(1-\chi(y)\big)y^{-1-\sigma}\,dy,\quad
R_1 := C_{1,\sigma}\!\int_2^{+\infty}\!\big(1-\chi(y)\big)y^{-1-\sigma}\,dy.
\]
Both constants $R_0$ and $R_1$ are finite and strictly positive, since $1-\chi\equiv1$ on $[3,+\infty)$ and $\theta<1<\sigma$. We may therefore choose $M>0$ as the unique value for which $R(0)=0$, i.e.
\begin{equation}\label{eq:cex-M}
M := \frac{A R_0}{R_1}.
\end{equation}
By linearity and \eqref{eq:frac-power}, \eqref{eq:cex-theta} and
\eqref{eq:cex-A} it follows that
\[
\chi(x)\,x^\gamma(-\Delta)^{\sigma/2}u(x)+|u'(x)|^m
= \big(A^m\theta^m-Ac_\sigma(\theta)\big)x^{\gamma+\theta-\sigma}+x^\gamma R(x) = x^\gamma R(x).
\]
Set
\begin{equation}\label{eq:cex-f-lip}
f(x):=x^\gamma R(x), \qquad x\in(0,2).
\end{equation}
Because $R\in C^\infty([0,2])$ vanishes at $0$, we have $R(x)=R'(0)x+O(x^2)$, whence
\begin{equation}\label{eq:cex-f-lip0}
f(x)=R'(0)\,x^{\gamma+1}+O(x^{\gamma+2}), \qquad x\to0^+.
\end{equation}
 
For $x\in[2,3]$, we have $u\in C^\infty([2,3])$ and bounded, so $(-\Delta)^{\sigma/2}u$ is smooth on $[2,3]$ by standard local regularity, with bounds depending only on $\theta,\sigma,A,\chi$. Set
\[
	f(x):=\chi(x)\,x^\gamma(-\Delta)^{\sigma/2}u(x)+|u'(x)|^m, \qquad x\in[2,3].
\]
The first term is smooth on $[2,3]$, and the only possible loss of regularity is carried by $|u'|^m$ at the zeros of $u'$. As $m>1$, $g(t):=|t|^m$ satisfies $g\in C^1(\R)$ with $g'(t)=m\,\mathrm{sgn}(t)|t|^{m-1}$ H\"older continuous of exponent $\min(1,m-1)$. Since $u'\in C^\infty([2,3])$, the composition $|u'(\cdot)|^m=g(u'(\cdot))$ inherits this regularity. Hence $|u'|^m\in C^{1,\min(1,m-1)}([2,3])$, and so does $f$. At $x=3$, all derivatives of $u$, hence of $u'$ and of $\chi x^\gamma(-\Delta)^{\sigma/2}u$, vanish, so $f$ and all its derivatives vanish as $x\to3^-$, matching $f\equiv0$ beyond. Therefore
\(
	f\in C^{1,\min(1,m-1)}([2,3]),
\)
in particular $f\in C^{0,1}([2,3])$, and $f\in C^\infty$ near every point where $u'\ne0$. 
 
\smallskip
\emph{Step 3: viscosity solution property at $x=0$.}
Since $\theta<1$, no smooth test function touches $u$ from above at $0$, so $u$ is trivially
a viscosity subsolution there. For the supersolution property, let $\phi\in C^2$ satisfy
$\phi(0)=0$ and $\phi\le u$ near $0$. Since $x_+^\gamma\big|_{x=0}=0$ and
$f(0)=0$ by \eqref{eq:cex-f-lip0},
\[
\chi(x)x_+^\gamma(-\Delta)^{\sigma/2}\phi(x)\Big|_{x=0}+|\phi'(0)|^m=|\phi'(0)|^m\ge0=f(0),
\]
which is exactly the required supersolution inequality at $0$.
 
\smallskip
\emph{Step 4: Lipschitz regularity of $f$.}
On $[2,3]$, $f\in C^{1,\min(1,m-1)}$ by Step~2, and $f\in C^\infty$ at point where $u' \neq 0$. 
Near $x=0$, the expansion \eqref{eq:cex-f-lip0} with $\gamma+1>1$ shows that
$f$ is differentiable at $0$ with $f'(0)=0$, and that
\[
f'(x)=(\gamma+1)R'(0)\,x^\gamma+O(x^{\gamma+1}), \qquad x\to0^+,
\]
is H\"older continuous of exponent $\gamma$. Hence $f\in C^{1,\gamma}$ near $x=0$: the choice
of $M$ in \eqref{eq:cex-M} cancels the leading coefficient $R(0)$ of the $x^\gamma$ term,
upgrading $f$ from $C^{0,\gamma}$ to $C^{1,\gamma}$ at the origin. Combining the three regimes,
\[
f\in C^{1,\min(\gamma,\,m-1)}(\R),
\]
in particular $f$ is globally Lipschitz continuous, while the solution $u$ remains exactly
$\theta$-H\"older and non-Lipschitz at $x=0$.
\end{proof}

\bigskip\subsection*{Acknowledgments}

 A. Ciomaga is partially supported by the ANR (Agence Nationale de la Recherche) through the COSS project ANR-22-CE40-0010 and by the Research Grant GAR2023 - code 73 - from the Donors' Recurrent Fund, at the disposal of the Romanian Academy and managed through the "PATRIMONIU" Foundation and by CNPq Grant 408169/2023-0.
T. M. L\^e is supported by the Austrian Science Fund (FWF) 10.55776/STA223. 
O. Ley is partially supported by the ANR (Agence Nationale de la Recherche) through the COSS project ANR-22-CE40-0010 and SABOCPR project ANR-25-CE40-3469-01, and by CNPq Grant 408169/2023-0.
E. Topp is partially supported by CNPq Grant 306022/2023-0, CNPq Grant 408169/2023-0, and FAPERJ APQ1 Grant 210.573/2024.

\bibliographystyle{plain}   
\bibliography{references}   

\end{document}